\documentclass{article}
\usepackage{t1enc,amsfonts,latexsym,amsmath,amsthm,verbatim,amssymb,graphics,mathdots,pstool,color}
\usepackage[utf8]{inputenc}
\usepackage{pgfplots}
\usepackage[showonlyrefs]{mathtools}
\usepackage{courier}
\usepackage{listings}
\usepackage[framed,numbered,autolinebreaks,useliterate]{mcode}
\usepackage{ amsmath,amsfonts}
\usepackage{graphicx}
\usepackage{mathrsfs}
\usepackage[export]{adjustbox}
\newcommand{\rank}{\mathsf{rank}}

\DeclareMathOperator*{\argmin}{arg\,min }
\DeclareMathOperator*{\argmax}{arg\,max }
\newcommand{\fro}[1]{\left\| #1\right\|_2^2}
\newcommand{\scal}[1]{\left \langle #1 \right \rangle}
\newcommand{\para}[1]{\left( #1\right)}

\newtheorem{theorem}{Theorem}[section]
\newtheorem{proposition}[theorem]{Proposition}

\newtheorem{corollary}[theorem]{Corollary}
\newtheorem{definition}  [theorem] {Definition}

\def\L{{\mathcal L}}
\def\T{{\mathcal T}}
\def\P{{\mathcal P}}

\def\M{{\mathcal{M}}}
\newcommand{\I}{\mathscr{N}}
\newcommand{\J}{\mathscr{J}}
\newcommand{\K}{\mathscr{K}}

\renewcommand{\H}{\mathcal{H}}
\newcommand{\N}{\mathscr{N}}
\title{Convergence of dual ascent in non-convex/non-differentiable optimization}
\author{Fredrik Andersson, Marcus Carlsson and Carl Olsson \thanks{Centre for Mathematical Sciences , Lund University, Box 118, SE-22100, Lund,  Sweden,  {fa,mc,calle@maths.lth.se}}}
\date{}

\begin{document}
\maketitle
\begin{abstract}
We revisit the classical dual ascent algorithm for minimization of convex functionals in the presence of linear constraints, and give convergence results which apply even for non-convex functionals. We describe limit points in terms of the convex envelope. We also introduce a new augmented version, which is shown to have superior convergence properties, and provide new results even for convex but non-differentiable objective functionals (as well as non-convex).

The results are applied to low rank approximation of a given matrix, subject to linear constraints. In particular, letting the linear constraints enforce Hankel structure of the respective matrices, the algorithms can be applied to complex frequency estimation. We provide numerical tests in this setting.

\end{abstract}

\section{Introduction and Motivation}\label{sec:intro}
A classical algorithm for solving linearily constrained convex optimization problems is the dual ascent scheme. Given a functional $\N(x)$ on some Hilbert space $\H$ and a set of linear constraints $\T(x)=b$, where $\T$ is a linear operator, the objective is to solve
\begin{equation}
\min_{\T(x)=b} \N(x).
\end{equation}
For simplicity we assume that $b=0$ since this can be achieved by translating the origin.
The condition $\T(x)=0$ then becomes equivalent to $x\in\M$ where $\M$ is the kernel of $\T$.

The dual ascent method considers the dual problem
$\max_\lambda g(\lambda),$
where the dual function is $g(\lambda) = \min_x \L_0(x,\lambda)$ and the Lagrangian is
\begin{equation}\label{lagr1}
\L_0(x,\lambda)=\N(x)+\langle \T(x),\lambda\rangle.
\end{equation}
The parameter $\lambda$ is the so called Lagrange multiplier and is an element in the codomain of $\T$. By introducing $\Lambda \in \M^\perp$ we may equivalently consider the restricted Lagrangian
\begin{equation}\label{lagr1}
\L(x,\Lambda)=\N(x)+\langle x ,\Lambda\rangle, \quad \Lambda \in \M^\perp.
\end{equation}
The dual ascent method tries to maximize the dual function by alternatively updating $x$ and $\Lambda$ according to
\begin{equation}\label{eq:dual_ascent_iter}
\begin{dcases}
x^{n+1} &= \argmin_{x}\N(x) + \langle \Lambda^n, x \rangle \\
\Lambda^{n+1} &= \Lambda^{n} + \alpha_k \P_{\M^\perp}(x^{n+1}),
\end{dcases}
\end{equation}
where $\P_{\M^\perp}$ denotes projection onto $\M^\perp$.
It can be seen that $x^{n+1}$ is a subgradient of $g$ at $\Lambda^n$ (see Section~\ref{sec2} for details), and the $\Lambda$-update can therefore be thought of as a projected gradient ascent step.

The main objective of this paper is to derive convergence results for cases where the objective function is non-differentiable and non-convex. While there exist general convergence results for the dual ascent approach these typically make additional smoothness assumptions (see Section~\ref{sec:related_work}). In contrast, the non-differentiable case remains relatively unexplored.
Our work is motivated by the problem of least squares low rank approximation with constraints.
Here
\begin{equation}\label{jf1}
\N(X) = \sigma_0^2\rank(X) +\|X-F\|_2^2,
\end{equation}
where $\sigma_0$ is a rank penalizing parameter, $F$ is a fixed matrix, typically related to measured data, and $\|\cdot\|_2$ denotes the Frobenius norm.
The constraint set $\M$ corresponds to some specific matrix structure, such as Hankel-form. Note that the $\rank$-function is constant everywhere except on a set of measure zero and non-convex. Furthermore, the sought optimum is typically of low rank and therefore located in the vicinity of discontinuities. This article contains convergence results that applies to the dual ascent scheme \eqref{eq:dual_ascent_iter} in this situation, distinguishing it from previous contributions listed in Section \ref{sec:related_work}.

\subsection{Contributions and Results}



The main contributions of our paper are twofold:
\begin{itemize}
\item We show convergence of the dual ascent updates \eqref{eq:dual_ascent_iter} for a general class non-convex (possibly discontinuous) objective functions $\N$ and describe the limit point in terms of the lower semi-continuous (l.s.c) convex envelope of $\N,$ denoted $\N^{**}$ .
\item  We propose a family of augmented formulations that can approximate the original problem arbitrarily well while exhibiting improved convergence properties. This also gives new results about convergence for non-differentiable convex functionals.
\end{itemize}
Our results are valid under the assumption of a so called \emph{feasible} objective
(see Definition~\ref{deffeas} for details). The functions $\N$ is feasible if it is
l.s.c, proper and $\N(x)$ grows sufficiently fast as $\|x\| \rightarrow \infty$.

Below we give simplified versions of our main results developed in Sections~\ref{secconv}~and~\ref{secconvII} respectively.

\begin{theorem}\label{t1intro}
Let $\N$ be feasible and consider the dual ascent scheme
\begin{align}
\label{dancintro1}x^{n+1} &= \argmin_{x}\N(x) + \langle \Lambda^n, x \rangle \\
\label{dancintro2}\Lambda^{n+1} &= \Lambda^{n} + \frac{1}{n+1} \P_{\M^\perp}(x^{n+1}),
\end{align}
with $\Lambda^0=0$. Suppose that the sequences $(x^{n})_{n=1}^\infty$ and $(\Lambda^{n})_{n=1}^\infty$ are bounded. Then there are convergent subsequences of  $(x^{n})_{n=1}^\infty$ and $(\Lambda^{n})_{n=1}^\infty$ with corresponding limits $x^\star$ and $\Lambda^\star$ satisfying; If $x^\star\in\M$, then it is a solution to \begin{equation}\label{convex_problemintro}\argmin_{x\in\M}\N^{**}(x).\end{equation} Moreover this happens whenever $$\argmin_{x}\N^{**}(x)+\scal{x,\Lambda^\star}$$ has a unique solution.
\end{theorem}
We remark that $1/(n+1)$ may be replaced by more general sequences, see Theorem \ref{t1} for more details.

The condition on boundedness clearly limits the applicability of the above result, but we will show that these are satisfied in the situation of low-rank approximation (Section \ref{sec:newsec3}). Moreover it would be desirable for the entire sequence to converge. The next theorem show that these issues disappear upon adding a small quadratic term to the objective functional. 

\begin{theorem}\label{t3intro}
Let $\alpha>0$ be fixed, let $\N$ be a finite-valued feasible functional, and consider the augmented dual ascent scheme
\begin{align}
\label{convex_augmented_da_intro1}x^{n+1} &= \argmin_{x} \N^{**}(x)+\scal{x,\Lambda^n} +\frac{\alpha}{2}\|x\|^2 \\
\label{convex_augmented_da_intro2}\Lambda^{n+1} &= \Lambda^{n} + \alpha\P_{\M^\perp}(x^{n+1}),
\end{align}
with $\Lambda^0=0$. It then holds that $(x^{n})_{n=1}^\infty$ and $(\Lambda^{n})_{n=1}^\infty$ converges to some limits $x^\star$ and $\Lambda^\star$, where $x^\star$ is the solution to \begin{equation}\label{convex_problem_augintro}\argmin_{x\in\M}\N^{**}(x)+\frac{\alpha}{2}\|x\|^2.\end{equation}
\end{theorem}
In fact, if $\N$ is convex to begin with, but non-differentiable, then $\N=\N^{**}$ and the above result is new even in this case.

Note that in contrast to our formulation \eqref{convex_augmented_da_intro1}, standard augmentation approaches \cite{bertsekas2014constrained} typically add the penalty term $\|\T x\|^2$, which
results in the augmented Lagrangian
\begin{equation}\label{lagr2}
\L_\alpha(x,\lambda) = \N^{**}(x)+\langle \T x,\lambda\rangle+\frac{\alpha}{2}\|\T x\|^2.
\end{equation}
The above expression is convex for a fixed $\lambda$. However, in the context of low rank approximation, where $x$ is an $N\times N$-matrix, $\T$ becomes an $N^2\times N^2$-matrix and the corresponding minimization over $x$ becomes very slow. In contrast, the $x$ update of \eqref{convex_augmented_da_intro1} has a closed form expression (in terms of the singular value decomposition) allowing rapid computation. The price one has to pay for this is that the $\frac{\alpha}{2}\|x\|^2$ affects the functional on the subspace $\M$, albeit negligibly assuming that $\alpha$ is small.

Theorem \ref{t3intro} is a combination of Theorem \ref{t3} and its corollaries. The full version also includes $\infty$-valued functionals. Moreover, Corollary \ref{c2} provides information on the speed of convergence of $(\Lambda^{n})_{n=1}^\infty$. To compute the update \eqref{convex_augmented_da_intro1}, explicit knowledge of $\N^{**}$ is needed, which is convex (but not necessarily differentiable). However, the (non-convex) updates \eqref{dancintro1} clearly arise as the limiting case of \eqref{convex_augmented_da_intro1} as $\alpha\rightarrow 0^+,$ so when $\N^{**}$ is known the augmented scheme \eqref{convex_augmented_da_intro1}-\eqref{convex_augmented_da_intro2} can be viewed as a minor modification of the original \eqref{dancintro1}-\eqref{dancintro2}, except for that the step-length in \ref{convex_augmented_da_intro2} is fixed. In section \ref{mixture} we present a crossover algorithm where this constraint is lifted, which in our numerical section \ref{sec:numeval} is shown to have superior performance.

In Section \ref{sec:newsec3} we consider the objective function \eqref{jf1}. We give explicit expressions for the convex envelope as well as formulas for the corresponding updates \eqref{dancintro1}-\eqref{dancintro2} (Proposition \ref{p1}) and \eqref{convex_augmented_da_intro1}-\eqref{convex_augmented_da_intro2} (Proposition \ref{p3}).  Concerning dual ascent, we prove that the sequences $(x^{n})_{n=1}^\infty$ and $(\Lambda^{n})_{n=1}^\infty$ always are bounded, so that Theorem \ref{t1intro} applies (see Theorem \ref{t2}). Concerning augmented dual ascent, Theorem \ref{t3intro} applies as stated above and yields that we always find a convergent sequence whose limit point solves \eqref{convex_problem_augintro} (see Theorem \ref{t4}).

\subsection{Related work}\label{sec:related_work}
The dual ascent algorithm goes back to the 60's \cite{everett1963generalized}, and the augmented version as well \cite{hestenes1969multiplier}, although other precursors relying on Lagrange multipliers are sometimes mentioned. We refer to \cite{boyd-etal-2011} or \cite{bertsekas2014constrained} for a more thorough overview of early results. Chapter 2 of the latter reference is devoted to the minimization of a non-convex functional $\N(x)$ under the constraint $\T(x)=0$, using dual ascent schemes similar to \eqref{dancintro1}-\eqref{dancintro2} and \eqref{convex_augmented_da_intro1}-\eqref{convex_augmented_da_intro2}, relying on Lagrangians of the type \eqref{lagr1}, but only local convergence results are provided, assuming that one starts near a minima with positive Hessian. As noted in Section~\ref{sec:intro} dual ascent is in fact a special case of the \emph{projected subgradient method}, see e.g. \cite{boyd2006subgradient}, Section 4.2 of \cite{bertsekas2011incremental} or Chapter 6 (available online only) of \cite{bertsekas2009convex}, for introduction and recent convergence results.

In the 90's these methods were extensively studied by Luo and Tseng \cite{luo1993convergence,luo1992linear,tseng1990dual}, mainly focusing on the convex situation and convergence under various smoothness assumptions on $\N$. They also began to study coordinate splitting methods \cite{luo1992convergence}, a work that was subsequently extended also to non-convex functionals \cite{tseng2001convergence}, albeit working with weaker concepts of convexity such as pseudoconvexity, quasiconvexity, hemivariate. A more recent contribution in this direction is \cite{goldstein2014fast}, which considers convergence of ADMM in the convex setting, a work that has inspired the proof of Theorem \ref{t3intro}.

There is a large body of work dealing with low rank approximation and optimization.
One of the earliest results is the Eckart-Young-Schmidt theorem \cite{schmidt1908,eckart-young-1936}
which gives a closed-form solution for the best least squares approximation of a given matrix $F$ with a matrix of specified maximal rank. More precisely, given a singular value decomposition (SVD) of
$F$ the best rank-$r$ approximation can be found by setting all but the first $r$ singular values to $0$.
In many applications (e.g. \cite{markovsky-2008} and the references therein) it is of interest to add additional constraints and penalties that model any additional prior knowledge we may have about the solution.
Adding any convex constraint to the formulation may seem like a minor change, however since the original formulation is not convex this makes the problem much more complicated and in practice iterative approaches (see e.g. \cite{lemmerling}) without optimality guarantees have to be applied.

In order to develop more flexible methods that allow incorporation of application specific priors researches have started to consider convex formulations for rank approximation. A popular heuristic is to replace the rank function with the so called nuclear norm \cite{recht-etal-siam-2010,candes-etal-acm-2011,olsson-oskarsson-scia-2009,angst-etal-iccv-2011}. Since this formulation penalizes all the singular values, not just the small ones, it has a shrinkage bias making it sensitive to high levels of noise and missing data (e.g. \cite{larsson2014rank,olsson-oskarsson-scia-2009}).

The original motivation for using the nuclear norm was given in
\cite{fazel-etal-acc-2015} where it is shown that the nuclear norm is the convex envelope of the rank function on the set $\{A; \sigma_1(A) \leq 1\}$, where $\sigma_1(A)$ is the largest singular value of $A$.
The constraint $\sigma_1(A) \leq 1$ is artificial and added since the convex envelope on the whole domain would simply be the zero function. Very recently it has been observed \cite{larsson2014rank,larsson2015convex} that a significantly stronger relaxation can be derived if one considers, not just the rank function, but also a least squares penalty term $\|A-F\|^2_2$.
The penalty term, which replaces the $\sigma_1(A) \leq 1$ constraint, effectively restricts the feasible domain to a neighborhood around $F$. As a consequence the obtained envelope, see \eqref{probmod}, is much more accurate in this region. In contrast to the nuclear norm it does not penalize singular values larger than $\sigma_0$ and therefore does not exhibit the same shrinking bias.

Because of the difficulty of achieving guaranteed global optimality local approaches are often employed.
If the rank of the sought matrix is known, bilinear parameterizations where the matrix is factored into $A=PL^T$ are used.
The easiest approach is to alternatively reestimate $P$ and $L$, e.g. \cite{markovsky-2008}. Buchanan and Fitzgibbon \cite{buchanan-fitzgibbon-cvpr-2005} showed that this method often exhibit very slow convergence. Instead they proposed a damped gauss-newton update that jointly optimizes over both factors. Hong and Fitzgibbon \cite{hong2015secrets} provide a
unified theoretical framework and experimental comparisons of many of these local factorization methods
with least squares residuals.

Finally, for the particular case of Hankel matrices and frequency estimation, the problem is severely non-convex with many local minima near the (constrained) global minima, see e.g. Figure 1 in \cite{lemmerling} and the surrounding discussion, or Section 3 in \cite{gillard}. To deal with this situation \cite{lemmerling,gillard} propose to also iterate over the initial guess. Section 4 of the (recent) article \cite{gillard} also contain a brief discussion of state of the art methods for the Hankel low rank approximation problem (sometimes called Hankel SLRA), where it is stated that ``None of these methods have the theoretical
property of convergence though and hence the construction of reliable methods for solving
the Hankel SLRA problem remains a wide open problem.''

\section{Preliminaries}\label{sec2}
In the entire paper $\H$ will denote a finite dimensional Hilbert space. By a proper fuctional we mean a $(-\infty,\infty]-$valued function on $\H$ which is not identically equal to $\infty$. Let $\J$ be a convex lower semi-continuous (l.s.c.) proper functional on $\H$.
The subdifferential of $\J$ is denoted $\partial \J(x)$ and consists of all $v \in \H$ that fulfill
\begin{equation}
\J(y) \geq \J(x) + \langle y-x,v \rangle
\end{equation}
for all $y \in \H$.
Similarly, in the concave case the subdifferential consists of all $v$ fulfilling the opposite inequality. A vector in $\partial \J(x)$ is denoted $\nabla \J(x)$.
By $\J^*$ we mean the Fenchel conjugate
\begin{equation}
\J^*(y) = \max_x \langle x,y \rangle - \J(x)
\end{equation}
(since there is no risk of confusion, we use this notation also for adjoints).
Note that the double Fenchel conjugate
$\J^{\ast \ast}$ is the l.s.c. convex envelope of $\J$ (see e.g. Proposition 13.39 \cite{bauschke2011convex}).

Let $\M\subset\H$ be a closed linear subspace and consider the problem \begin{equation}\label{probrank3}
\argmin_{x\in\M} \J(x).
\end{equation}
Introduce an orthonormal basis $\{E_j\}_{j=1}^{\dim\M}$ for $\M$, and similarly an orthonormal basis $\{\tilde{E}_j\}_{j=1}^{\dim \M^\perp}$ for $\M^\perp$. If $\T:\H\rightarrow \mathbb{C}^{\dim\M^\perp}$ is the operator whose output $\T(x)$ is the coefficients of $x$ in the basis $\{\tilde{E}_j\}_{j=1}^{\dim \M^\perp}$, then \eqref{probrank3} can be written as \begin{equation}\label{probrank2}
\argmin_{\T(x)=0} \J(x),
\end{equation}
(in fact, any operator with $\mathsf{Ker} \T=\M$ would do). Following the method of multipliers (see e.g. \cite{boyd2006subgradient}, Chapter 6 and 7), we introduce the Lagrangian
$$\J(x)+\langle \T(x),\lambda\rangle$$ where $\lambda\in \mathbb{C}^{\dim\M^{\perp}}$ is the so called Lagrange multiplier and the scalar product is the canonical one in $\mathbb{C}^{\dim\M^{\perp}}$. Since $\langle \T(x),\lambda\rangle=\langle x,\T^*(\lambda)\rangle$, and $\mathsf{Ran} \T^*=(\mathsf{Ker} \T)^{\perp}=\M^\perp$, we may equivalently consider the \textit{restricted Lagrangian}
\begin{equation}\label{m}
\L(x,\Lambda) =  \J(x) + \langle \Lambda, x \rangle, \quad \Lambda \in \M^\perp,
\end{equation}
where $\Lambda$ replaces $\T^*(\lambda)$. The dual function is then \begin{equation}\label{defg}g(\Lambda)=\min_x \L(x,\Lambda)=-\J^\ast(-\Lambda).\end{equation}  In particular, $g$ is concave.
The dual ascent consists in applying the projected subgradient method to the function $g$ and the subset $\M^{\perp}$, (see \cite{boyd-etal-2011} Section 2.1 and \cite{boyd2006subgradient}, Chapter 6 and 7). More precisely, if we let $\P_{\M^\perp}(x)$ denote the projection of $x$ onto the subspace $\M^\perp$, i.e. $\P_{\M^\perp}(x)=\sum_{j=1}^{\dim\M^\perp}\langle x,\tilde{E}_j\rangle \tilde{E}_j$, then the dual ascent algorithm for  \eqref{probrank2} reads
\begin{equation}\label{da}
\begin{dcases}
x^{n+1} &= \argmin_{x}\J(x) + \langle \Lambda^n, x \rangle \\
\Lambda^{n+1} &= \Lambda^{n} + \alpha_n \P_{\M^\perp}(x^{n+1}),
\end{dcases}
\end{equation}
for some sequence of predetermined parameters $\alpha_n$, (see e.g. Ch. 6 and 7 in \cite{boyd2006subgradient}), and $\Lambda^0=0$. To better see the connection between the updates and the subgradient of $g$, we remind the reader that \begin{equation}
\label{sd}y\in\partial \J(x)\Leftrightarrow \J(x)+\J^\ast(y)=\scal{x,y}\Leftrightarrow x\in\partial\J^\ast(y)
\end{equation}
(see e.g. \cite{bauschke2011convex}, Theorem 16.23). In terms of $g$ and $\Lambda$, \eqref{sd} reads
\begin{equation}
\label{sd1}-\Lambda\in\partial \J(x)\Leftrightarrow \J(x)+\scal{x,\Lambda}=g(\Lambda)\Leftrightarrow x\in \partial g(\Lambda).
\end{equation}
Since clearly $-\Lambda^n\in\partial \J(x^{n+1})$ in \eqref{da}, it follows that $x^{n+1}\in\partial g(\Lambda^n)$ and hence the update for $\Lambda$ can be rewritten
$$\Lambda^{n+1} = \Lambda^{n} +\alpha_n \P_{\M^\perp}\nabla g(\Lambda^{n}).$$ The iterates thus ascend on the concave hill $g$, at least if the step length $\alpha_n$ is not too big.

\section{Non-convex cost functionals}\label{sec4}
Suppose now that $\N$ is a non-convex proper l.s.c. functional that we wish to minimize over some subspace $\M$ of the finite dimensional Hilbert space $\H$.
One may then still attempt the dual ascent minimization scheme,
\begin{align}\label{danc1}
x^{n+1} &= \argmin_{x}\N(x) + \langle \Lambda^n, x \rangle \\
\label{danc2}\Lambda^{n+1} &= \Lambda^{n} + \alpha_n\P_{\M^\perp}(x^{n+1}),
\end{align}
albeit with little hope of being able to prove convergence or, in case it converges, proving that we have found a global minimum.

Since $\N^{***}=\N^*$ (see e.g. Prop. 13.14 in \cite{bauschke2011convex}) we have
\begin{equation}
\max_x \N(x)+\langle \Lambda,x\rangle = -\N^*(-\Lambda)
= -\N^{***}(-\Lambda) = \max_x \N^{**}(x)+\langle \Lambda,x\rangle.
\label{eq:dual_func_equality}
\end{equation}
That is, the dual functions obtained using $\N$ and its convex envelope $\N^{**}$ respectively, coincide. Furthermore, it is a simple observation that \begin{equation} \label{IR_minimum} \argmin_{x}\N(x) + \langle \Lambda, x \rangle \subset \argmin_{x}\N^{\ast \ast}(x) + \langle \Lambda, x \rangle. \end{equation}
Hence the sequence generated by \eqref{danc1}-\eqref{danc2} is contained among the sequences that can be generated by a dual ascent scheme using $\N^{\ast \ast}$. This makes it plausible that it should in fact converge to an optimum of $\min_{x\in \M}\N^{\ast \ast}(x)$.

While $\N(x)$ and $\N(x)^{**}$ typically attain the same minimal values it is important to realize that if we restrict the functionals to $\M$ this may no longer hold.
The convex envelope over $\M$, i.e. $(\N|_{\M})^{\ast\ast}$, does not necessarily coincide with $\N^{\ast\ast}$ restricted to ${\M}$, (i.e. $\N^{\ast\ast}|_{\M}$), see Figure \ref{convex_envelope_illustration}. The latter is clearly a convex function which lies below $\N$, so we do have $$(\N|_{\M})^{\ast\ast}\geq (\N^{\ast\ast})|_{\M}.$$ We will give conditions under which the dual ascent scheme \eqref{danc1}-\eqref{danc2} converges, and show that the convergence point is then typically a solution of $(\N^{\ast\ast})|_{\M}$.
Note that computing $(\N|_{\M})^{\ast\ast}$ is just as difficult as solving the original non-convex problem since
\begin{equation}
(\N|_\M)^*(0) = \max_{x \in \M} \langle x,0 \rangle - \N(x) = - \min_{x \in \M} \N(x).
\end{equation}
and therefore not feasible to work with.

\subsection{Feasible functionals}

In order for the above discussion to make sense, we have to assume that $\N$ grows fast enough in all directions so that the above minimizers exist. We therefore restrict attention to the following class of functions, (in the notation of \cite{bauschke2011convex}, functionals that satisfy \eqref{lok} are called supercoercive); \begin{definition}\label{deffeas}
A l.s.c. proper functional $\N$ will be called feasible if it is bounded below and \begin{equation}\label{lok}\lim_{\|x\|\rightarrow\infty}\frac{\N(x)}{\|x\|}=\infty.\end{equation}
\end{definition}

\begin{figure*}\label{fig1}
\centering
\includegraphics[width=0.8\linewidth, trim=0.0cm -0.5cm 0.0cm 0.0cm]{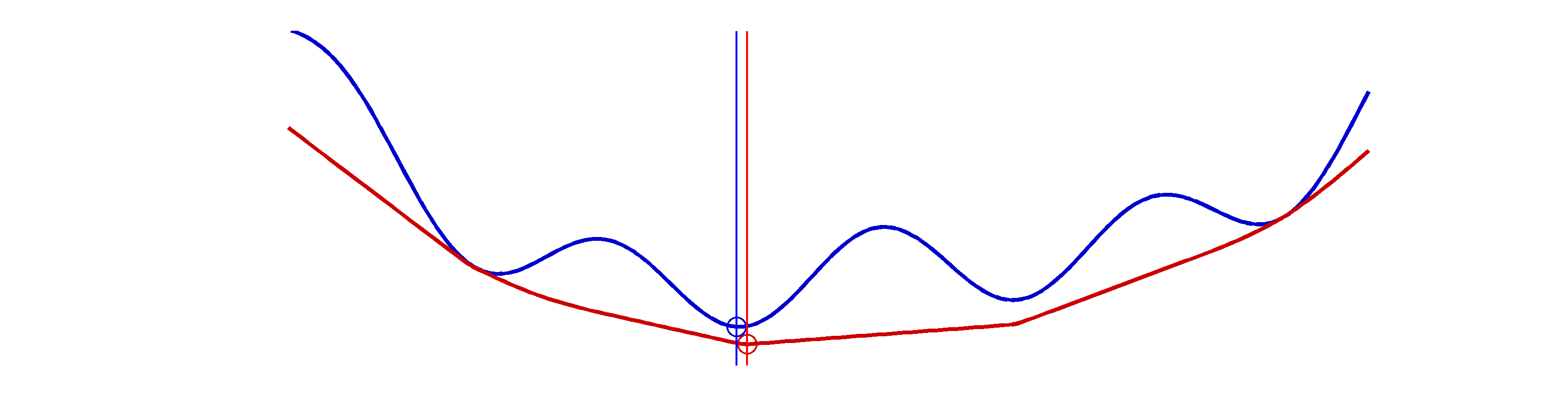}\\
\includegraphics[width=0.48\linewidth, trim=1.5cm 1.5cm 1.5cm 1.5cm]{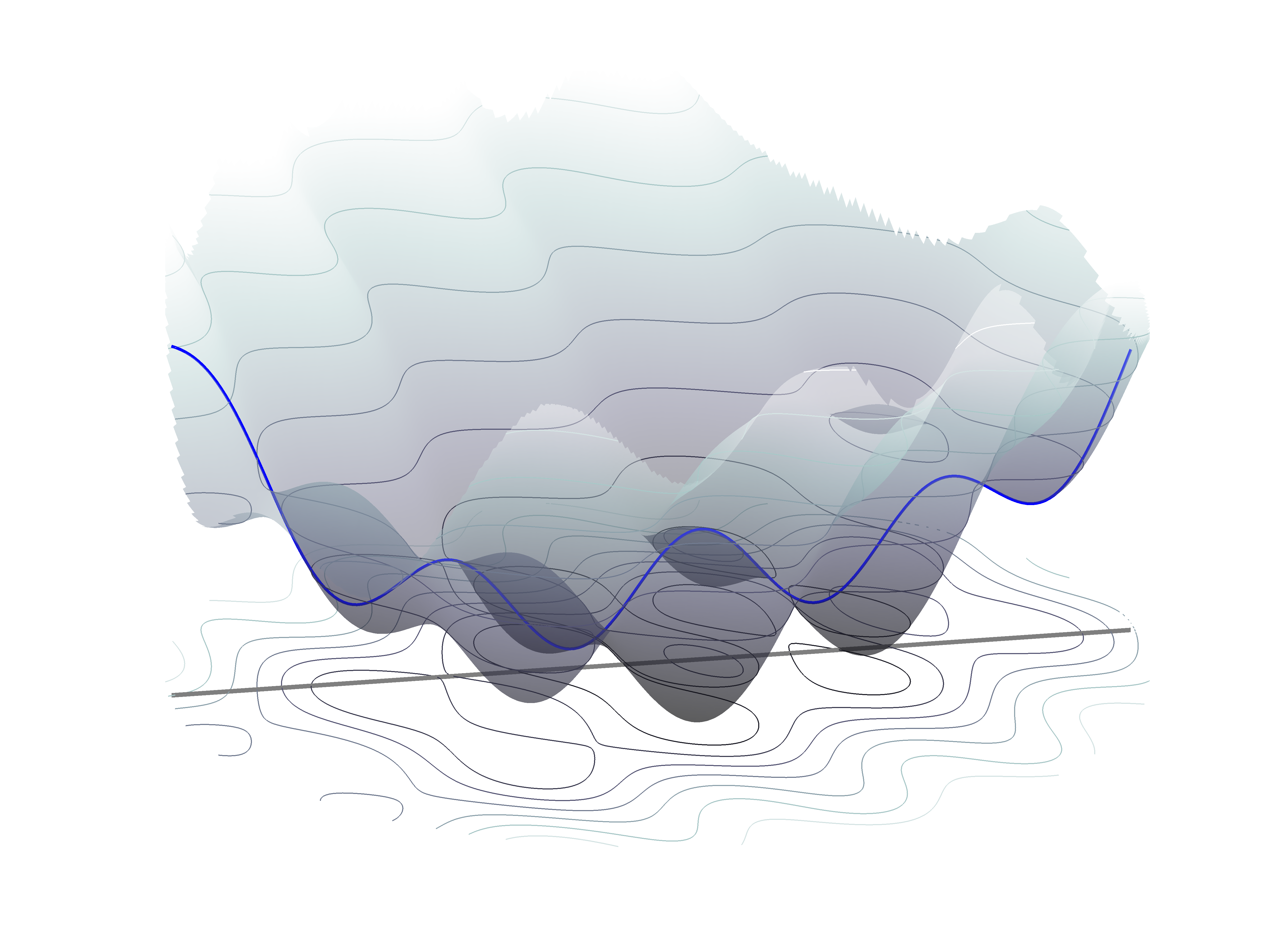}
\includegraphics[width=0.48\linewidth, trim=1.5cm 1.5cm 1.5cm 1.5cm]{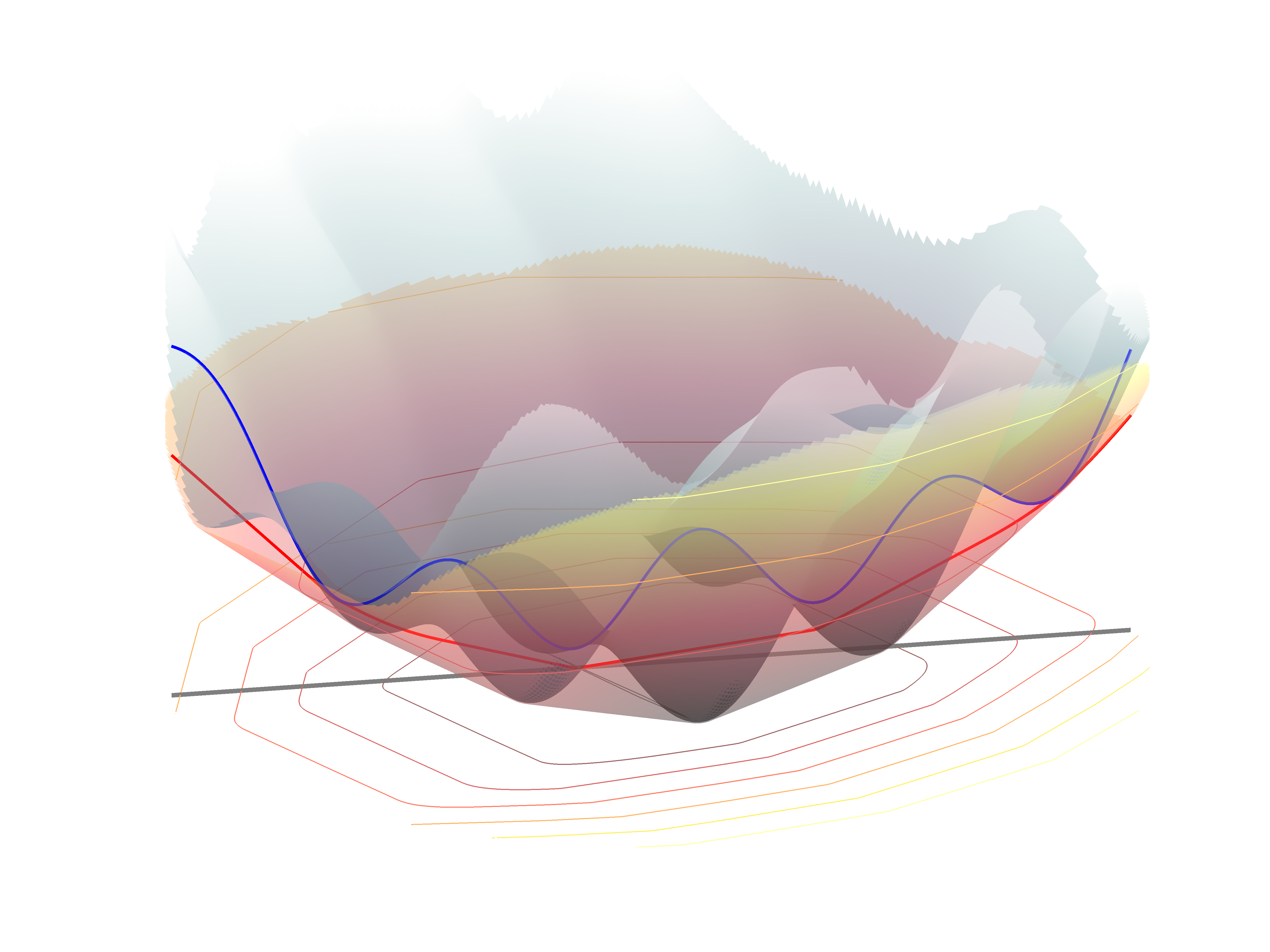}
\caption{Illustration of the constrained optimization. The bottom left panel shows a non-convex function along with its level sets. The gray line represents the constraint, and the blue curve the values of the constrained function. The bottom right panel shows the same setup, but here the convex envelope is shown as well in red/yellow. The values of the convex envelope along the constraint is shown in the red curve. The top figure show a one-dimensional plot of the values of the original function (blue) and the convex envelope (red) evaluated on the constraint. The respective minima are shown by circles and highlighted by the vertical lines. Note that they are located close to each other, but that they are not identical despite the fact that the global minimum for the original function and its convex envelope coincide. \label{convex_envelope_illustration}}
\end{figure*}

In the remaining part of this section we establish some results concerning feasibility of conjugate and double conjugate functions that will be useful later on. The next proposition summarize properties of $\N^*$.

\begin{proposition}\label{p2}
If $\N$ is feasible then $\N^*$ is a continuous convex functional which is bounded below.
Furthermore, both $\N^*$ and its subdifferentials are bounded on bounded subsets of $\H$.
Finally, if $\N$ is finite-valued, then $\N^*$ is feasible. In particular it then has bounded level sets.
\end{proposition}
\begin{proof}
That $\N^*$ is bounded below follows by the trivial estimate $\N^*(\Lambda)\geq -\N(0)$ for all $\Lambda$. We now consider bounds from above.

Given $r\geq 0$ set $\mu(r)=\inf_{\|x\|=r}\N(x)$, which by assumption is bounded below and satisfies $\lim_{r\rightarrow\infty}\mu(r)/r=\infty$. Fix $R>0$ and suppose that $\Lambda$ satisfies $\|\Lambda\|\leq R$.
Then $$\N^*(\Lambda)=\sup_x \scal{x,\Lambda}-\N(x)\leq \sup_r r\left(R-\frac{\mu(r)}{r}\right).$$
Since the second term of the product will be negative for sufficiently large $r$,
the supremum is clearly finite and independent of $\Lambda$ (as long as $\|\Lambda\|\leq R$), so $\N^*$ is bounded on bounded subsets of $\H$. That $\N^*$ is convex is well-known and easy to check, (see e.g. Proposition 13.11 in \cite{bauschke2011convex}). Since a locally bounded convex function is continuous (
Corollary 8.30 in \cite{bauschke2011convex}) it follows that $\N^*$ is continuous. Theorem 16.17 in the same reference also yields the statement concerning subdifferentials.

Now let $\N$ be finite-valued and suppose that $\N^*$ is not supercoercive. Pick a sequence $(\Lambda^n)_{n=1}^\infty$ such that $$\limsup_{n\rightarrow\infty} \frac{\N^*(\Lambda^n)}{\|\Lambda^n\|}=c<\infty$$ and $\lim_{n\rightarrow\infty}\|\Lambda_n\|=\infty$. Since the unit ball is compact in finite dimensional Hilbert spaces, there exists a subsequence $(\Lambda^{n_k})_{k=1}^\infty$ such that $\Lambda^{n_k}/\|\Lambda^{n_k}\|$ converges to some $\Gamma$ with $\|\Gamma\|=1$. By the Fenchel-Young inequality (Proposition 13.13 \cite{bauschke2011convex}), we then have
$$\limsup_{k\rightarrow\infty} \frac{\N^*(\Lambda^{n_k})}{\|\Lambda^{n_k}\|}\geq \limsup_{k\rightarrow\infty} \frac{\scal{2c\Gamma,\Lambda^{n_k}}-\N(2c\Gamma)}{\|\Lambda^{n_k}\|}=2c, $$
which is a contradiction.
\end{proof}

The bounded level sets of $\N^*$ will be an important condition in subsequent results, and it is therefore important to realize that this property may fail even in the convex case, if $\N$ assumes the value $\infty.$ Pick for example $\H=\mathbb{R}$ and $$\N(x)=\left\{\begin{array}{ll}
                                                                                                    x^2/2 & x\geq 0 \\
                                                                                                    \infty & x<\infty
                                                                                                  \end{array}\right.
$$ for which
$$\N^*(x)=\left\{\begin{array}{ll}
                                                                                                    x^2/2 & x\geq 0 \\
                                                                                                    0 & x<\infty
                                                                                                  \end{array}\right.
$$
In the next proposition we collect properties of $\N^{**}$.

\begin{proposition}\label{p4}
If $\N$ is feasible, then so is $\N^{**}$.
\end{proposition}
\begin{proof}
$\N^{**}$ is l.s.c since it is the Fenchel conjugate of $\N^*$. As already noted, $\N^{**}$ is the l.s.c. convex envelope of $\N$, and hence it is bounded below since $\N$ is, and proper since $\N^{**}\leq \N$. To show that $\N^{**}$ satisfies \eqref{lok}, let $c>0$ be arbitrary and pick $R$ such that $\N(x)\geq c\|x\|$ whenever $\|x\|\geq R$. Let $B$ be a bound from below for $\N$ and note that $$\N(x)\geq \max(B,B+c(\|x\|-R)).$$
Since the right hand side side is a convex lower bound on $\N$ it follows that $$\limsup_{\|x\|\rightarrow\infty}\frac{\N^{**}(x)}{\|x\|}\geq \limsup_{\|x\|\rightarrow\infty}\frac{\max(B,B+c(\|x\|-R))}{\|x\|}=c$$
from which the desired conclusion is immediate, as $c$ was arbitrary.
\end{proof}


\section{Convergence of dual ascent}\label{secconv}

Let $\N$ be a non-convex proper l.s.c. functional on some (finite dimensional) Hilbert space $\H$, that we wish to minimize over a subspace $\M$. In this section we provide a general convergence result for the dual ascent scheme \eqref{danc1}-\eqref{danc2}. This result only gives convergence of a subsequence, which can be remedied by considering an augmented variation of the algorithm. This is done in Section \ref{secconvII}.

Let $(x^{n})_{n=1}^\infty$ and $(\Lambda^n)_{n=1}^\infty$ be given by the dual ascent scheme \eqref{danc1}-\eqref{danc2} and set
\begin{equation}\label{tgen}n_k=\argmax_{0\leq n\leq k}-\N^*(-\Lambda^n).\end{equation}

\begin{theorem}\label{t1}
Let $\alpha_n$ be a sequence satisfying $0< \alpha_n\leq 1$ and \begin{equation}\label{ak}\lim_{N\rightarrow\infty}{\sum_{n=1}^N\alpha_n^2}\Big/{\sum_{n=1}^N\alpha_n}=0.\end{equation}
Let $\N$ be feasible and suppose that the sequences $(x^{n})_{n=1}^\infty$ and $(\Lambda^{n})_{n=1}^\infty$ (given by \eqref{danc1}-\eqref{danc2}) are bounded. Given any convergent subsequences $(x^{n_{k_j}})_{k=1}^\infty$ and $(\Lambda^{n_{k_j}})_{k=1}^\infty$ with limits $x^\star$ and $\Lambda^\star$, we have that \begin{equation}\label{convex_problem4}\N^{**}(x^\star)+\scal{x^\star,\Lambda^{\star}}=\inf_{x\in\M}\N^{**}(x).\end{equation} Moreover, if $\argmin_{x}\N^{**}(x)+\scal{x,\Lambda^\star}$ has a unique solution, then $x^\star$ lies in $\M$ and is a solution to \begin{equation}\label{convex_problem}\argmin_{x\in\M}\N^{**}(x).\end{equation}
\end{theorem}

\begin{proof}
As noted in Section \ref{sec4}, the updates are the same as if we apply dual ascent to the convex l.s.c. functional $\J=\N^{**}$, (recall \eqref{da}). Note that $\J^*=\N^{***}=\N^*$ (see e.g. Prop. 13.14 in \cite{bauschke2011convex}), so that the dual function $g$ introduced in \eqref{defg} becomes \begin{equation}\label{defgrepeat}g(\Lambda)=\min_x \J(x)+\scal{x,\Lambda}=-\N^\ast(-\Lambda).\end{equation}Note that the sequence $(g(\Lambda^{n_k}))_{k=1}^\infty$ is non-decreasing by the definition of the numbers $n_k$.
As we observed in Section \ref{sec2}, the updates for $\Lambda^n$ are those one gets by applying the projected subgradient method to $g$ and $\M^\perp$. Since $(\Lambda^{n})_{n=1}^\infty$ is a bounded sequence and $g$ has uniformly bounded subdifferentials in any ball of a fixed radius (Proposition \ref{p2}), it follows by the results in \cite{boyd2006subgradient} (see in particular equation (2)), that $\left(g(\Lambda^{n_k})\right)_{k=1}^\infty$ converges to the maximum value of $g$, as long as $(\alpha_n)_{n=1}^\infty$ satisfies the stated conditions.

Since $(\Lambda^{n_k})_{k=1}^\infty$ and $(x^{n_k})_{k=1}^\infty$ are bounded sequences they have convergent subsequences (by Alaoglu's theorem). Let $(k_j)_{j=1}^\infty$ be any strictly increasing sequence such that both $(\Lambda^{n_{k_j}})_{j=1}^\infty$ and $(x^{n_{k_j}})_{j=1}^\infty$ converge, and denote the corresponding limits by $\Lambda^\star$ and $x^\star$.
By Proposition \ref{p2} $g$ is continuous and therefore $g(\Lambda^\star)$ is the maximum of $g$. We also claim that  \begin{equation}\label{gt}\J(x^\star)+\scal{x^\star,\Lambda^\star}=g(\Lambda^\star),\end{equation} which clearly follows if we show that $\lim_{j\rightarrow\infty}\J(x^{n_{k_j}})=\J(x^\star)$. This in turn is clear by the identity $\J(x^{n})=g(\Lambda^{n})-\scal{x^{n},\Lambda^{n}}$, which holds for all $n$, and the continuity of $g$.

The minimum of \begin{equation}\label{5}\J(x)+\scal{x,\Lambda^\star}\end{equation} is attained for some $\tilde x$ on $\M$,
for otherwise standard arguments show that a different value of $\Lambda$ would yield a higher value of $g$ (see e.g. Sections 5.2.3 and 5.4.2 in \cite{boyd2004convex}). By definition this minimum equals $g(\Lambda^\star)$ and since $\Lambda^{\star}\in \M^\perp$ it follows that $$\J(\tilde x)=\min_{x\in\M}\J(x)=\min_{x\in\H}\J(x)+\scal{x,\Lambda^\star}=g(\Lambda^\star)=\J(x^\star)+\scal{x^\star,\Lambda^\star},$$
where the last identity follows by \eqref{gt}. With this we have established \eqref{convex_problem4}.

The above chain of equalities also implies $$\J(\tilde x)+\scal{\tilde x,\Lambda^\star}=\J(\tilde{x})=\J(x^\star)+\scal{x^\star,\Lambda^\star}$$ so if \eqref{5} only has one minimizer we immediately deduce that $x^\star=\tilde{x}$, and the remainder of the theorem follows.
\end{proof}

\begin{figure}[htb]
\begin{center}
\includegraphics[width=40mm,valign=c]{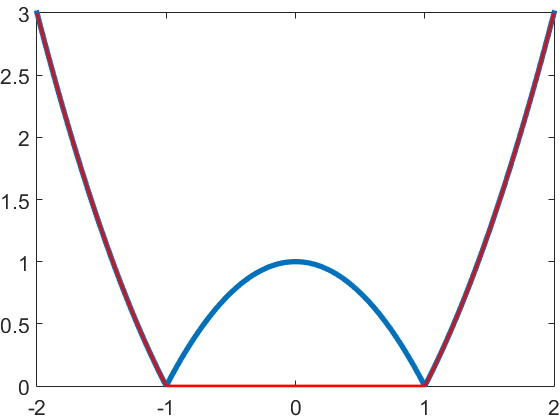}
\includegraphics[width=40mm,valign=c]{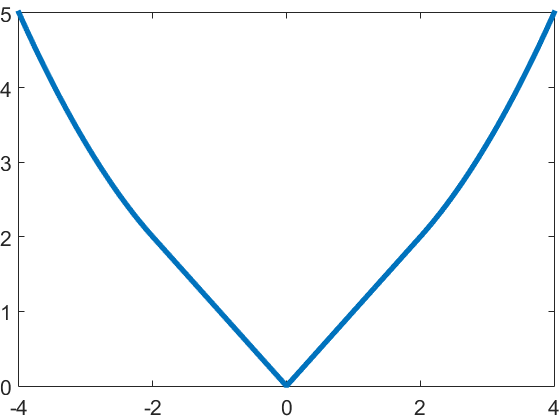}
\begin{minipage}[c]{30mm}
\begin{tabular}{c|c|c}
$n$ & $x^n$ & $\Lambda^n$ \\
\hline
0 & - & 0 \\
1 & 1 & 1 \\
2 & -1 & 1/2 \\
3 & 1 & 1/6 \\
4 & -1 & -1/12 \\
5 & 1 & 7/60
\end{tabular}
\end{minipage}
\end{center}
\caption{\emph{Left} - The function $\N(x) = |x^2-1|$ (blue) and its convex envelope (red).
\emph{Middle} - The conjugate function $\N^*(\Lambda)$.
\emph{Right} - Sequence generated by dual ascent.}
\label{fig:nonuniqueexample}
\end{figure}
When \eqref{convex_problem} does not have a unique solution there may not be any subsequence that converges to an $x\in \M$. As an example, consider the case of $\H=\mathbb{R}$, $\M=\{0\}$, $\N(x)=|x^2-1|$ and $\alpha_n=\frac{1}{n+1}$. Then $\M^{\perp}=\mathbb{R}$ so $\Lambda^{n+1}=\Lambda^n+\frac{1}{n+1}x^{n+1}$.
It is easily computed that $\N^*(\Lambda)=|\Lambda|$ for $|\Lambda|\leq 2$, and $\N^*(\Lambda)=1+\Lambda^2/4$ elsewhere. Clearly $\N^{**}(x)=\max(0,x^2-1)$.
The functions $\N$,$\N^*$ and $\N^{**}$ are displayed in Figure~\ref{fig:nonuniqueexample}
together with a sequence generated by dual ascent for this problem.
It is clear that $(x^n)_{n=1}^\infty$ will oscillate between values in $\{\pm 1\}$ whereas $\lim_{n\rightarrow\infty}\Lambda^n=0$. Thus, depending on which subsequence is chosen, $x^\star$ will be either 1 or -1, and in neither case does it lie on $\M$. To avoid the possibility of such undesirable sequences, we now consider an augmented version.
\section{Convergence of augmented dual ascent}\label{secconvII}

In this section we will make a minor change of the algorithm and give a different proof, inspired by \cite{goldstein2014fast}, leading us to assume instead that $\alpha_n$ is fixed. We first consider the case when $\K$ is a \textit{convex} feasible functional on some (finite dimensional) Hilbert space $\H$ and $\M$ a subspace, and discuss the important particular case $\K=\N^{**}$ in Section \ref{sec:JN}. We introduce the augmented dual ascent scheme
\begin{align}\label{convex_augmented_da1}
x^{n+1} &= \argmin_{x} \K(x)+\scal{x,\Lambda^n} +\frac{\alpha}{2}\|x\|^2 \\
\label{convex_augmented_da2}\Lambda^{n+1} &= \Lambda^{n} + \alpha\P_{\M^\perp}(x^{n+1}),
\end{align}
with $\Lambda^0=0$. To our best knowledge the following result, which assumes no differentiability of $\K$, is new.

\begin{theorem}\label{t3}
Let $\alpha>0$ be fixed and let $\K$ be a convex feasible functional. Suppose that $(\Lambda^{n})_{n=1}^\infty$, given by \eqref{convex_augmented_da2}, is a bounded sequence. It then holds that $(x^{n})_{n=1}^\infty$ and $(\Lambda^{n})_{n=1}^\infty$ converges to some limits $x^\star$ and $\Lambda^\star$, where $x^\star$ is the solution to \begin{equation}\label{convex_problem_aug}\argmin_{x\in\M}\K(x)+\frac{\alpha}{2}\|x\|^2.\end{equation}
\end{theorem}

Remark; the condition on $(\Lambda^{n})_{n=1}^\infty$ is fulfilled whenever $\K$ is finite-valued, see the corollaries.

\begin{proof}
Set $\J(x)=\K(x)+\frac{\alpha}{2}\|\P_\M x\|^2$, 
let $g$ be as in Section \ref{sec2} and define $h:\H\rightarrow\mathbb{R}$ by \begin{equation}\label{brase}h(\Lambda)=g(\P_{\M^\perp}\Lambda)=-\J^*(-\P_{\M^\perp}\Lambda)=\min_x \J(x)+\scal{x,\P_{\M^\perp}\Lambda}.\end{equation} Since $\K$ is feasible it directly follows that $\J$ is feasible. Proposition \ref{p2} thus implies that $\J^*$ is bounded below which implies that $h$ is bounded above. Also note that, since the sequence $(\Lambda^{n})_{n=1}^\infty$ lies in $\M^{\perp}$ by definition, see \eqref{convex_augmented_da2}, there is no difference between $h$ and $g$ for these points. Clearly \begin{equation}\label{ineq167} \partial h(\Lambda)=\P_{\M^\perp}\partial g(\P_{\M^\perp}\Lambda)\end{equation} and from \eqref{sd1} we know that $x\in\partial g(\Lambda)$ if and only if $-\Lambda\in\partial\J(x)$. Note that $x^{n+1}$ is given by $$x^{n+1}=\argmin_x \J(x)+\scal{x,\Lambda^n}+\frac{\alpha}{2}\fro{\P_{\M^{\perp}}x},$$ i.e. $$\nabla\J(x^{n+1})=-(\Lambda^n+\alpha\P_{\M^{\perp}}x^{n+1})=-\Lambda^{n+1},$$ and hence we conclude that
\begin{equation}
\label{6}x^{n+1}=\nabla g(\Lambda^{n+1}).
\end{equation}
Combining \eqref{convex_augmented_da2} and \eqref{ineq167}, we see that \eqref{6} implies the following;
\begin{equation}\label{ineq5}
h(\Gamma)\leq h(\Lambda^{n+1})+\langle\Gamma-\Lambda^{n+1},\nabla^+h(\Lambda^{n+1})\rangle_2= h(\Lambda^{n+1})+\alpha^{-1}\langle\Gamma-\Lambda^{n+1}, \Lambda^{n+1}-\Lambda^n\rangle_2
\end{equation}
where $\Gamma$ is arbitrary. Setting $\Gamma=\Lambda^n$ immediately gives
\begin{equation}\label{ineq3}
h(\Lambda^{n+1})\geq h(\Lambda^{n})+\alpha^{-1}\|\Lambda^{n+1}-\Lambda^{n}\|^2
\end{equation}
so $h$ increases for each iteration. Thus $\lim_{n\rightarrow\infty}h(\Lambda^n)=c$ exists, and since $h$ is bounded above it is finite.
We now show that this is actually the supremum of $h$. First note that \begin{equation}\label{7}\langle\Gamma-\Lambda^{n+1},\Lambda^{n+1}-\Lambda^{n}\rangle_2=\|\Gamma-\Lambda^n\|^2-\|\Gamma-\Lambda^{n+1}\|^2-\|\Lambda^n-\Lambda^{n+1}\|^2.\end{equation} If $\Gamma$ is a point such that $h(\Gamma)\geq c$, then the above is positive by \eqref{ineq5}, which means that $\|\Gamma-\Lambda^n\|$ is decreasing with increasing $n$. Moreover, combining \eqref{ineq5} with \eqref{7} we get that the (positive) sum
$$\alpha\sum_{n=1}^\infty \big(h(\Gamma)-h(\Lambda^n)\big)=\|\Gamma-\Lambda^1\|^2-\lim_{n\rightarrow\infty}\|\Gamma-\Lambda^n\|^2-\sum_{n=1}^\infty\|\Lambda^n-\Lambda^{n+1}\|^2$$
is bounded by $\|\Gamma-\Lambda^{1}\|^2$, which shows that $\lim_{n\rightarrow\infty} h(\Lambda^n)=h(\Gamma)$. By recalling that this also equals $c$, it follows that $h(\Gamma)>c$ can not hold for any point. We conclude that $c$ is indeed the supremum of $h$.

Since $(\Lambda^{n})_{n=1}^\infty$ is a bounded sequence (and $\H$ is finite dimensional) it has a convergent subsequence $(\Lambda^{n_k})_{k=1}^\infty$, whose limit we denote by $\Lambda^\star$. By \eqref{brase} and the fact that $\J^*$ is l.s.c, it follows that $h$ is u.s.c. and so $h(\Lambda^\star)\geq\lim_{k\rightarrow\infty}h(\Lambda^{n_k})=c$, so $h(\Lambda^\star)=c$. Let $\Gamma=\Lambda^\star$ in \eqref{ineq5}. Fix $n$ and let $n_k$ be a number in the subsequence such that $n_k>n$. Similarly to above we then have
\begin{equation}\label{8}\alpha\sum_{j=n}^{n_k} \big(h(\Lambda^\star)-h(\Lambda^j)\big)=\|\Lambda^\star-\Lambda^{n}\|^2-\|\Lambda^\star-\Lambda^{n_k}\|^2-\sum_{j=n}^{n_k}\|\Lambda^j-\Lambda^{j+1}\|^2\end{equation}
which in the limit gives
$$\sum_{j=n}^{\infty}\|\Lambda^j-\Lambda^{j+1}\|^2+\alpha\sum_{j=n}^{\infty} \big(h(\Lambda^\star)-h(\Lambda^j)\big)=\|\Lambda^\star-\Lambda^{n}\|^2.$$
Both sums are convergent since their summands are positive. It follows that we can make $\|\Lambda^\star-\Lambda^{n}\|^2$ arbitrarily small upon choosing $n$ sufficiently large, (and also that this quantity is decreasing). It follows that $\lim_{n\rightarrow \infty}\Lambda^n=\Lambda^\star$, as desired.

It remains to prove that $(x^{n})_{n=1}^\infty$ converges to a minimum of \eqref{convex_problem_aug}. The operator which is implicitly defined in \eqref{convex_augmented_da1} (taking $\Lambda^n$ to $x^{n+1}$) equals the proximal mapping of $\frac{1}{\alpha}\K$ evaluated at $\frac{1}{\alpha}\Lambda^n$  Since $\K$ is convex and l.s.c., these mappings are firmly non-expansive, and in particular continuous (see e.g. Ch. 4.1 and 12.4 in \cite{bauschke2011convex}, in particular Proposition 12.27). The convergence of $(x^{n})_{n=1}^\infty$ thus follows from that of $(\Lambda^{n})_{n=1}^\infty$. Finally, if $x^\star$ would not be a solution to \eqref{convex_problem_aug}, then we can argue as in the end of Theorem \ref{t1} to get a contradiction (see \eqref{5} and use $\J(x)=\K(x)+\frac{\alpha}{2}\|x\|^2$, which has a unique minimizer since it is strictly convex).
\end{proof}
We collect a few results that came out in the above proof.
\begin{corollary}\label{c1}
Let $h$, defined as in the above proof, attain its supremum. Then $(\Lambda^{n})_{n=1}^\infty$, given by \eqref{convex_augmented_da2}, is a bounded sequence. In particular, this happens if $\K$ is a feasible finite-valued functional.
\end{corollary}

Remark; the functional on $\mathbb{R}$ given by $f(x)=\infty$ for $x<0$, $f(0)=0$ and $f(x)=x\log x-x$ has $f^*(y)=e^y$ and is convex (since it equals $f^{**}$). By modification of this simple example, one can show that there are situations where $h$ does not attain its supremum.

\begin{proof}
The first part of the corollary follows immediately by the sentence following \eqref{7}. If we now assume that $\K$ is finite valued then so is $\J$, and Proposition \ref{p2} implies that $\J^*$ is feasible. By \eqref{brase} we conclude that $h$ attains its supremum.

\end{proof}

We now consider the speed of convergence of the dual variable.

\begin{corollary}\label{c2}
Suppose that $(\Lambda^{n})_{n=1}^\infty$ is bounded. Then $(\Lambda^{n})_{n=1}^\infty$ is Fej\'{e}r monotone with respect to the set of maximizers of $h$, which is non-empty. Moreover, if $c$ is the supremum of $h$ then $c-h(\Lambda^n)=o(1/n)$.
\end{corollary}
\begin{proof}
The existence of maximizers was established in the proof of Theorem \ref{t3}. Let $\Gamma$ be one such. It follows from \eqref{7} and the subsequent argument that $(\|\Gamma-\Lambda^{n}\|)_{n=1}^\infty$ is decreasing, i.e. that $(\Lambda^{n})_{n=1}^\infty$ is Fej\'{e}r monotone with respect to the set of maximizers of $h$. Now recall that $h(\Lambda^n)$ increases by \eqref{ineq3}. In analogy with  \eqref{8} we can get the concrete estimate
$$(m-n)\alpha\big(h(\Gamma)-h(\Lambda^m)\big)\leq\alpha\sum_{j=n}^{m} \big(h(\Gamma)-h(\Lambda^j)\big)\leq \|\Gamma-\Lambda^n\|^2,$$ showing that $$\limsup_{m\rightarrow\infty}m(h(\Gamma)-h(\Lambda^m))\leq \limsup_{m\rightarrow\infty} \frac{m}{m-n}\frac{\|\Gamma-\Lambda^n\|^2}{\alpha}=\frac{\|\Gamma-\Lambda^n\|^2}{\alpha}.$$ Since $n$ is arbitrary and $\|\Gamma-\Lambda^n\|\rightarrow 0$ as $n\rightarrow\infty$, the limsup is actually 0, and hence the speed of convergence of $h(\Lambda^n)$ is at least $o(1/n)$.
\end{proof}

\subsection{Other steplengths}\label{mixture}

The augmented dual ascent scheme has a drawback; the term $\frac{\alpha}{2}\|x\|^2$ is included for convergence purposes and will influence the optimal point, and hence it is desirable to keep $\alpha$ small. However, $\alpha$ is also the step length in the updates of $\Lambda^n$, and in order for the algorithm to converge rapidly one would like to have $\alpha$ fairly large. To remedy this one may consider the following scheme
\begin{align}\label{convex_augmented_daNEW1}
x^{n+1} &= \argmin_{x} \K(x)+\scal{x,\Lambda^n} +\frac{\alpha}{2}\|x\|^2 \\
\label{convex_augmented_daNEW2}\Lambda^{n+1} &= \Lambda^{n} + \alpha_n \P_{\M^\perp}(x^{n+1}),
\end{align}
where $\lim_{n\rightarrow\infty}{\alpha_n}=\alpha$, which is closer to the original scheme \eqref{danc1}-\eqref{danc2}. Upon assuming further regularity of $\K$, it is possible to modify the proof of Theorem \ref{t3} and derive conditions under which this scheme converges to $$\argmin_{x\in\M} \K(x) +\frac{\alpha}{2}\|x\|^2,$$ but we refrain from this and content with noting that the results of the previous section applies whenever $\alpha_n=\alpha$ for all sufficiently large $n$ (which is immediate by Theorem \ref{t3}). We will however use the scheme \eqref{convex_augmented_daNEW1}-\eqref{convex_augmented_daNEW2} in our numerical section.

\subsection{The case $\J=\N^{**}$}\label{sec:JN}

We now move back to considering a non-convex feasible functional $\N$, as in Section \ref{secconv}. In this section we consider the augmented dual ascent scheme
\begin{align}\label{convex_augmented_daN1}
x^{n+1} &= \argmin_{x} \N^{**}(x)+\scal{x,\Lambda^n} +\frac{\alpha}{2}\|x\|^2 \\
\label{convex_augmented_daN2}\Lambda^{n+1} &= \Lambda^{n} + \alpha\P_{\M^\perp}(x^{n+1}),
\end{align}
with $\Lambda^0=0$, which is an alteration of \eqref{danc1}-\eqref{danc2}. As we shall see, the addition of the quadratic penalty may alter the limit point slightly. However, for small values of $\alpha$ it will give a good approximation. (In the context of low rank approximation we investigate this further in Section \ref{sec:numeval}.) The penalty term allows us to prove convergence of the generated sequence itself without the need for examining sub-sequences. However, note that it does require explicit knowledge of $\N^{**}$.

\begin{theorem}\label{t33}
Let $\alpha>0$ be fixed and let $\N$ be a feasible functional on some finite dimensional Hilbert space $\H$. Suppose that that $\N$ is finite-valued, or that $(\Lambda^{n})_{n=1}^\infty$, given by \eqref{convex_augmented_da2}, is a bounded sequence. It then holds that $(x^{n})_{n=1}^\infty$ and $(\Lambda^{n})_{n=1}^\infty$ converges to some limits $x^\star$ and $\Lambda^\star$, where $x^\star$ is the solution to \begin{equation}\label{convex_problem_aug}\argmin_{x\in\M}\N^{**}(x)+\frac{\alpha}{2}\|x\|^2.\end{equation}
\end{theorem}
\begin{proof}
The result is immediate by applying Theorem \ref{t3} and Corollary \ref{c1} to $\K=\N^{**}$, since this is feasible and/or finite-valued whenever $\N$ is, by Proposition \ref{p4} and the basic inequality $\N^{**}\leq \N$.
\end{proof}

\section{Low rank approximation with constraints}\label{sec:newsec3}
We now depart from the general theory and consider the problem of low rank approximation with subspace constraints
where the objective function $\N$ is of the form
\begin{equation}
\N_F(X) = \sigma_0\rank(X)+\|X-F\|^2_2,
\label{eq:rankobj}
\end{equation}
where $X$ is an $M\times N$ matrix.

Before we address the optimization schemes we present some useful results.
For the above formulation \eqref{eq:rankobj} the conjugate function $\N_F^*$ and the convex envelope $\N_F^{**}$ can be computed in closed form  \cite{larsson-olsson-ijcv-2016}.
The conjugate function is given by
\begin{equation}
\label{fenchel1}\N_F^\ast (\Lambda) = \sum_j \max_{}\para{\sigma_j^2\left(\frac{\Lambda}{2}+F\right)-\sigma_0^2,0} -\fro{F}
\end{equation}
and the convex envelope is given by
\begin{equation} \label{probmod}
\N_F^{\ast \ast}(X)= \sum_j\left( \sigma_0^2-\para{\max_{}\para{\sigma_0-\sigma_j(X),0}}^2\right) + \|X-F\|_2^2.
\end{equation}
In the primal update of augmented dual ascent \eqref{convex_augmented_daN1} we need to solve problems of the form
\begin{equation}
\argmin_X \N_F^{**}(X)+\frac{\alpha}{2} \|X\|_2^2.
\label{eq:augmentedprimal}
\end{equation}
It turns out that this can be efficiently computed by modifying the singular values of $F$. More precisely, suppose that $F$ has the singular value decomposition $F=U\Sigma_{\phi} V^*$,
where $\Sigma_\phi$ is a diagonal matrix containing the singular values $\phi=(\phi_j)_{j=1}^{\min(M,N)}$. Given a function $f:[0,\infty)\rightarrow\mathbb{C}$, we  introduce the operator
\begin{equation}\label{singval}
\mathfrak{S}_f(F)=U\Sigma_{f(\phi)} V^*,\end{equation}
that modifies $F$ by changing the singular values from $\phi_j$ to $f(\phi_j)$. This operation is known as the ``singular value functional calculus'' \cite{andersson2015operator} as well as ``generalized matrix function'' \cite{arrigo2015computation}.
The following proposition now shows how to solve \eqref{eq:augmentedprimal}.

\begin{proposition}\label{p3}
Let $\alpha\geq 0$ and set \begin{equation*}f_{\alpha}(x)=\left\{\begin{array}{ll}
                   0 & x<{\sigma_0}\\
                   \frac{2}{\alpha}(x-\sigma_0) & {\sigma_0}\leq x < {(1+\frac{\alpha}{2})}\sigma_0\\
                   {(1+\frac{\alpha}{2})}^{-1}x, & {(1+\frac{\alpha}{2})}\sigma_0\leq x.
                 \end{array}\right.
\end{equation*}
Then $\mathfrak{S}_{f_{\alpha}}(F)$ solves \eqref{eq:augmentedprimal}.
\end{proposition}

\begin{proof}
By von-Neumann's inequality both problems are solved by a matrix of the form $X=U\Sigma_{\sigma} V^*$, (for a detailed version of this inequality, which also provides the above information on the singular vectors at optimum, see e.g. \cite{de1994exposed}). Therefore the singular values are found by minimizing  $$\sigma\mapsto\sigma_0^2-\para{\max_{}\para{\sigma_0-\sigma,0}}^2+(\sigma-\phi_j)^2+\frac{\alpha}{2}\sigma^2.$$ Differentiation of this expression shows that the minimum given by $f_\alpha(\phi_j)$.
\end{proof}
Note that for $\alpha > 0$ the objective is strictly convex and the minimum is therefore unique.
The next proposition characterizes all the solutions for the case $\alpha = 0$.

\begin{proposition}\label{p1}
Let $F=U\Sigma_{\phi} V^*$. If $\alpha=0$ all solutions of \eqref{eq:augmentedprimal} are of the form $U\Sigma_\sigma V^*$, where $\sigma$ is given by
\begin{equation}\label{4}\left\{\begin{array}{ll}
           \sigma_j=\phi_j & \phi_j>\sigma_0 \\
           \sigma_j=\mu_j & \phi_j=\sigma_0 \\
           0 & \phi_j<\sigma_0
         \end{array}\right.
\end{equation}
and $\mu_j$ is a free parameter which can be chosen in $0\leq \mu_j\leq \sigma_0$.
The solution of the non-convex problem \eqref{eq:rankobj} is also of the form \eqref{4}, but here $\mu_i$ must be chosen to be either $\sigma_0$ or $0$.
In particular, both \eqref{eq:rankobj} and \eqref{eq:augmentedprimal} are solved by $\mathfrak{S}_{f_{0}}(F)$
\end{proposition}

\begin{proof}
Again, von-Neumann's inequality implies that both problems are solved by a matrix of the form $X=U\Sigma_{\sigma} V^*$. To chose the $\sigma_j$'s in the case of $\I^{\ast\ast}$, we need to minimize the functional $$\sigma\mapsto\sigma_0^2-\para{\max_{}\para{\sigma_0-\sigma,0}}^2+(\sigma-\phi_j)^2,$$
and it is easy to see that the solutions are as stated in \eqref{4}. The corresponding statement for $\I$ is even simpler, we omit the details. The final statement is obtained by setting $\mu_j=\sigma_0$ whenever there is ambiguity, i.e. $\phi_j=\sigma_0$.
\end{proof}
Figure~\ref{fig:falpha} shows examples of $f_\alpha(x)$ when $\sigma_0 = 1$ for different values of alpha. Note that in all cases $f_\alpha$ will set singular values less than $\sigma_0$ to zero.
For $\alpha = 0$ singular values larger than $\sigma_0$ remain unaffected by $f_\alpha$, while for $\alpha > 0$ these are subjected to a penalty.
We also remark that the so called "hard thresholding" performend by $f_0$ gives a minimizer of both $\N_F$ and $\N^{**}_F$ although it may not be unique in either case.
\begin{figure}[htb]
\begin{center}
\includegraphics[width=50mm]{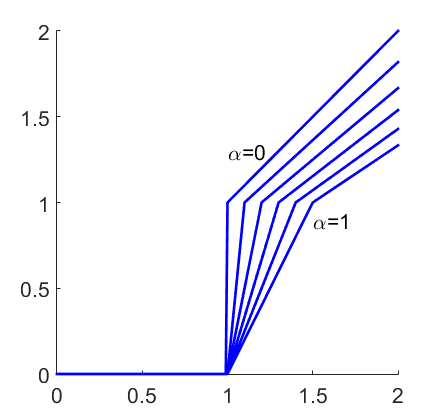}
\end{center}
\caption{The function $f_\alpha(x)$ for $\alpha = 0,0.2,...,1$, and $\sigma_0=1$.}
\label{fig:falpha}
\end{figure}

\subsection{Dual ascent}\label{subsec1}

The dual ascent algorithm \eqref{danc1}-\eqref{danc2} for trying to minimize
\eqref{eq:rankobj}
reads
\begin{equation}\label{nonconvex_da}
\begin{dcases}
X^{n+1} &= \argmin_{X} \sigma_0^2\rank(X) +\|X-F\|_2^2 +\langle \Lambda^n, X \rangle \\
\Lambda^{n+1} &= \Lambda^{n} + \alpha_n\P_{\M^\perp}(X^{n+1}),
\end{dcases}
\end{equation}
The Lagrange multiplier can easily be absorbed into the quadratic term, since \begin{equation}\label{g7}\|X-F\|_2^2+\langle \Lambda, X \rangle =\fro{X-\para{F-\frac{\Lambda}{2}}}+ \langle \Lambda, F \rangle-\fro{\frac{\Lambda}{2}},\end{equation} and $\langle \Lambda, F \rangle-\fro{\frac{\Lambda}{2}}$ is independent of $X$.
Using \eqref{g7}, the scheme becomes;
\begin{equation}\label{nonconvex_darepeat}
\begin{dcases}
X^{n+1} &=\mathfrak{S}_{f_0}\para{F-\frac{\Lambda^n}{2}}\\
\Lambda^{n+1} &= \Lambda^{n} + \alpha_n\P_{\M^\perp}(X^{n+1}),
\end{dcases}
\end{equation}
with $\Lambda^0=0$.
By the remarks in Section \ref{sec4} (as well as Proposition \ref{p1}), it is also a dual ascent scheme for the minimization of $\I_{F}^{\ast\ast}$ over $\M$.

Similar to Section~\ref{secconv} we let
\begin{equation}\label{tg}n_k=\argmax_{0\leq n\leq k}-\I^*_{F}(-\Lambda^n).\end{equation}

\begin{theorem}\label{t2}
Let $\alpha_n$ be a sequence satisfying $0< \alpha_n\leq 1$ and \eqref{ak}. Let $(X^n)_{n=1}^\infty$ and $(\Lambda^n)_{n=1}^\infty$ be given by \eqref{nonconvex_darepeat}, and $n_k$ by \eqref{tg}. It then holds that $(X^{n_k})_{k=1}^\infty$ and $(\Lambda^{n_k})_{k=1}^\infty$ have convergent subsequences. Moreover, for their limits, $X^\star$ and $\Lambda^\star$, we have that if $F-\Lambda^\star/2$ has no singular value equal to $\sigma_0$, then $X^\star$ is a solution to $$\argmin_{X\in\M} \I_{F}^{**}(X).$$
\end{theorem}

\begin{proof}
It is clear that $\I_{F}$ is feasible so Theorem \ref{t1} applies.
The fact that both sequences $(X^{n})_{n=1}^\infty$ and $(\Lambda^{n})_{n=1}^\infty$ are bounded is shown separately in Proposition \ref{l1} below. Assuming this, Alaoglu's theorem implies that both sequences $(X^{n_k})_{k=1}^\infty$ and $(\Lambda^{n_k})_{k=1}^\infty$ have convergent subsequences. For the final statement, it follows by \eqref{g7} that $$\I_{F}^{\ast\ast}+\scal{X,\Lambda^{\star}}=\I_{F-\Lambda^\star/2}^{\ast\ast}(X)+const,$$ and hence Theorem \ref{t1} and Proposition \ref{p1} together imply that the corresponding minimization problem has a unique solution if and only if $F-\Lambda^\star/2$ has no singular value equal to $\sigma_0$.

\end{proof}

\begin{proposition}\label{l1}
If $0<\alpha_n\leq 1$ for all $n\in\mathbb{N}$, then the sequences $(X^n)_{n=1}^\infty$ and $(\Lambda^n)_{n=1}^\infty$ given by \eqref{nonconvex_darepeat}, are bounded.
\end{proposition}
\begin{proof}
The sequence $(X^n)_{n=1}^\infty$ is clearly bounded if $(\Lambda^n)_{n=1}^\infty$ is. To prove that $(\Lambda^n)_{n=1}^\infty$ is bounded, first note that \begin{equation}\label{1}\Lambda^{n+1}=\Lambda^n+\alpha_n \P_{\M^\perp}\mathfrak{S}_{f_0}\left(F-\frac{\Lambda^n}{2}\right).
\end{equation}
Since $\Lambda^{n}\in\mathsf{Ran} \P_{\M^\perp}$ and $\P_{\M^\perp}$ is self-adjoint, we also have \begin{align*}&\scal{\Lambda^{n}, \P_{\M^\perp}\mathfrak{S}_{f_0}(F-\frac{\Lambda^n}{2})}_2=\scal{\Lambda^{n}, \mathfrak{S}_{f_0}(F-\frac{\Lambda^n}{2})}_2=\scal{\Lambda^n,F-\frac{\Lambda^n}{2}}_2+\scal{\Lambda^{n}, (\mathfrak{S}_{f_0}-I)(F-\frac{\Lambda^n}{2})}_2.\end{align*}
The operator $I-\mathfrak{S}_{f_0}$ keeps only the singular values lower than $\sigma_0$, and sets the other ones to 0. Thus $\|(\mathfrak{S}_{f_0}-I)(F-\frac{\Lambda^n}{2})\|_2^2\leq K\sigma_0^2$, where $K=\min(M,N)$ is the amount of singular values. Setting $R=\|{\Lambda^n}\|_2$, the Cauchy-Schwartz inequality thus gives \begin{equation}\label{2}|\scal{\Lambda^{n},\P_{\M^\perp}\mathfrak{S}_{f_0}(F-\frac{\Lambda^n}{2})}|\leq R\|{F}\|_2-\frac{R^2}{2}+\sqrt{K}\sigma_0 R.\end{equation}
Moreover, since both $\P_{\M^\perp}$ and $\mathfrak{S}_{f_0}$ are contractions we have
\begin{equation}\label{3}\|\P_{\M^\perp}\mathfrak{S}_{f_0}(F-\frac{\Lambda^n}{2})\|_2^2\leq \|F-\frac{\Lambda^n}{2}\|_2^2\leq \frac{R^2}{4}+\|F\|_2R+\|F\|_2^2.\end{equation}
Combining \eqref{1}-\eqref{3} and recalling that $0<\alpha_n\leq 1$, we get
\begin{align*}\|\Lambda^{n+1}\|_2^2&\leq R^2+2\alpha_n(R\|{F}\|_2-\frac{R^2}{2}+\sqrt{K}\sigma_0 R)+\alpha_n^2(\frac{R^2}{4}+\|F\|_2R+\|F\|_2^2)= \\& (1-\frac{\alpha_n}{2})^2R^2+2\alpha_n(R\|{F}\|_2+\sqrt{K}\sigma_0 R)+\alpha_n^2(\|F\|_2R+\|F\|_2^2)\leq (1-\frac{\alpha_n}{2}) R^2+\alpha_n(c_1 R+c_2)\end{align*}
where $c_1=3\|{F}\|_2+2\sqrt{K}\sigma_0$ and $c_2=\|F^2\|_2$.
Note that these constants are independent of $n$. Setting $p(R) = - \frac{R^2}{2}+c_1 R+c_2$ our inequality can be written $$\|\Lambda^{n+1}\|_2^2\leq R^2+\alpha_n p(R).$$

Whenever $p(R)\leq 0$ we clearly have
$\|\Lambda^{n+1}\|_2\leq R$.
Now suppose that $p$ is not negative everywhere.
Then it has two real roots, and takes positive values between them.
Let us denote the larger root by $R_0$ and the maximal value by $p_{\text{max}}$.
Recall that $R=\|\Lambda^n\|_2$. For $R\leq R_0$ the quantity $\|\Lambda^{n+1}\|_2$ can now be uniformly bounded by
the constant $\sqrt{R_0^2 + \alpha_n p_{\text{max}}}$, and for $R> R_0$ the earlier inequality reads $\|\Lambda^{n+1}\|_2\leq \|\Lambda^n\|_2$. Summing up, we have shown that $$\|\Lambda^{n+1}\|_2\leq\max(\sqrt{R_0^2+\alpha_n p_{\text{max}}},\|\Lambda^{n}\|_2),$$ from which it clearly follows that $(\Lambda^n)_{n=1}^\infty$ is a bounded sequence.
\end{proof}

\subsection{Augmented dual ascent}\label{subsec2}

By \eqref{g7} and Proposition \ref{p3} it follows that the augmented dual ascent scheme \eqref{convex_augmented_daN1}-\eqref{convex_augmented_daN2} takes the form
\begin{equation}\label{nonconvex_ada}
\begin{dcases}
X^{n+1} &=\mathfrak{S}_{f_{\alpha}}\para{F-\frac{\Lambda^n}{2}},\\
\Lambda^{n+1} &= \Lambda^{n} + \alpha\P_{\M^\perp}(X^{n+1}),
\end{dcases}
\end{equation}
with $\Lambda^0=0$.
Note that the primal update of the (un-augmented) dual ascent scheme \eqref{nonconvex_darepeat} is the limiting case as $\alpha\rightarrow 0$, which has been claimed earlier in the paper. Theorem \ref{t3} and Corollary \ref{c1} immediately gives

\begin{theorem}\label{t4}
The sequences $(X^{n})_{n=1}^\infty$ and $(\Lambda^{n})_{n=1}^\infty$ converge to some limits $X^\star$ and $\Lambda^\star$, where $X^\star$ is the solution to \begin{equation}\label{SLRA_aug}\argmin_{X\in\M}\I_{F}^{**}(X)+\frac{\alpha}{2}\|X\|_2^2.\end{equation}
\end{theorem}

\section{Numerical Evaluation}\label{sec:numeval}

In this section we evaluate the proposed methods in the context of rank minimization with Hankel constraints. An example of an application is the decomposition of a signal into complex exponentials. There is a well known connection between the rank of a Hankel matrix and the number of exponentials needed for the generating function of the Hankel matrix, usually referred to as \emph{Kronecker's theorem} \cite{Kronecker}; Given a complex valued vector $f$ with elements
\begin{equation}\label{Kronecker}
f(j) = \sum_{p=1}^{P} c_p e^{\zeta_p j}, \quad c_p,\zeta_p\in \mathbb{C}, \quad -N \le j \le N,
\end{equation}
the Hankel matrix $F$ generated by the vector $f$ (i.e., the elements of $F$ satisfy
$
F(j,k)=f(j+k-N-1)$) is of rank $P$ (with the exception of degenerate cases), and conversely, if a Hankel matrix has rank $P$ then its generating vector is of the form \eqref{Kronecker} (again with the exception of degenerate cases), see \cite{ACgeneraldomain}.

\subsection{Convergence Evaluation}
We first perform a quantitative evaluation of the convergence of our algorithms.
We compare three approaches:
\begin{description}
\item[DA] - The dual ascent scheme \eqref{nonconvex_darepeat} with the step-sizes $\alpha_n = \frac{1}{n+1}$.
\item[ADA] - The augmented dual ascent scheme \eqref{nonconvex_ada} with a step-size $\alpha$.
\item[mod-ADA] - The augmented dual ascent scheme where we replace the fixed step-size
$\alpha$ with $\alpha_n = \frac{2}{(n+1)^2}+\alpha$ as in \eqref{convex_augmented_daNEW1}-\eqref{convex_augmented_daNEW2}. This allows the algorithm to take large steps in the beginning of the optimization, which as we shall see greatly accelerates convergence.
\end{description}
To create low rank Hankel matrices we make use of \emph{Kronecker's theorem} and randomly select sums of exponentials. For the results displayed in Figure~\ref{fig:alphaconvergence} we used
\begin{equation}
f(t) = \sum_{i=1}^4 a e^{bt}\cos(10ct+d\pi),
\label{eq:cossum}
\end{equation}
where $a$ and $d$ are uniformly distributed over $[0,1]$ and  $b$ and $c$ belong to a normal
distribution with mean zero and standard-deviation one.
We sampled the function in $200$ equally spaced points between $-1$ and $1$ and formed a
$101 \times 100$ Hankel matrix. Since each term in the sum \eqref{eq:cossum} consists of two complex exponentials the resulting matrix, which we use as ground truth, will have rank 8.
To generate the measurement matrix we added Gaussian noise with $0.1$ standard-deviation to each element of the ground truth matrix.

Figure~\ref{fig:alphaconvergence} illustrates the convergence of the three methods.
Here we solved $100$ instances of the problem and plotted average primal and dual objective values for the first $100$ iterations.
In each iteration we generated a primal feasible solution by projecting the current estimate $X^n$ onto the closest Hankel matrix $X = \P_\M(X^n)$. To compute primal objective values we then used
$\N^{**}_F(X)$ for DA and $\N^{**}_F(X)+\frac{\alpha}{2}\|X\|_2^2$ for both ADA and mod-ADA,
since this is what the methods will converge to.

To evaluate dual objectives we used conjugates of the above primal functions. For DA the dual is simply $-\N^*_F(-\Lambda^n)$ with $\N_F^{**}$ as in \eqref{fenchel1}. For ADA and mod-ADA the conjugate $\left(\N^{**}_F(X)+\frac{\alpha}{2}\|X\|_2^2\right)^*$ can be computed by noting that
$$
 \langle X,Y \rangle - \N^{**}_F(X)-\frac{\alpha}{2}\|X\|_2^2 =
-\left(\N^{**}_{F+\frac{Y}{2}}(X)-\frac{\alpha}{2}\|X\|_2^2\right)+\left\|F+\frac{Y}{2}\right\|_2^2-\|F\|_2^2.
$$
Inserting $Y=-\Lambda^n$ and maximizing the right hand side with respect to $X$ shows that $X^n$ optimal. Therefore we get the dual objective function
$$
\left(\N^{**}_{F-\frac{\Lambda^n}{2}}(X^n)-\frac{\alpha}{2}\|X^n\|_2^2\right)-\left\|F-\frac{\Lambda^n}{2}\right\|_2^2+\|F\|_2^2.
$$

In Figure~\ref{fig:alphaconvergence} we ran the experiment twice, first with $\alpha = 0.1$ and then with $\alpha=0.001$. For $\alpha=0.1$ the methods seem to converge relatively fast
(see Figure~\ref{fig:alphaconvergence}, left).
For $\alpha = 0.001$ (Figure~\ref{fig:alphaconvergence}, right) the convergence of ADA becomes prohibitively slow since it is forced to very take small gradient steps.
In contrast the variable step-size of mod-ADA still generates good solutions in very few iterations.
\begin{figure}[htb]
\begin{center}
\includegraphics[width=60mm]{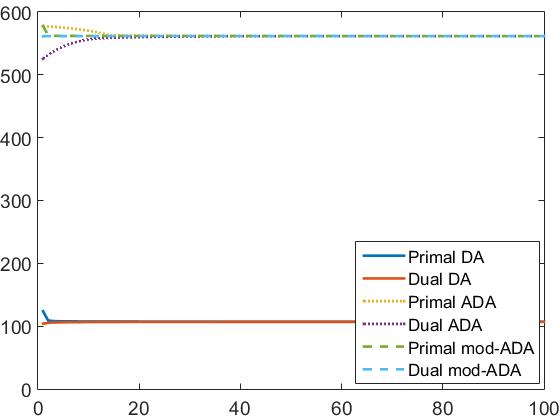}
\includegraphics[width=60mm]{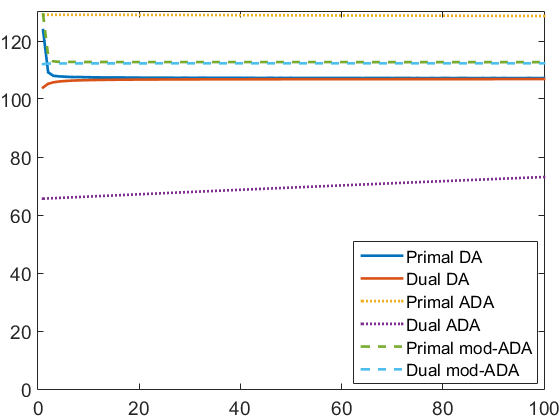}
\end{center}
\caption{Average primal and dual objectives vs. iterations for the three formulations.
Left $\alpha=0.1$ right $\alpha=0.001$.}
\label{fig:alphaconvergence}
\end{figure}

\begin{figure}[htb]
\begin{center}
\includegraphics[width=90mm]{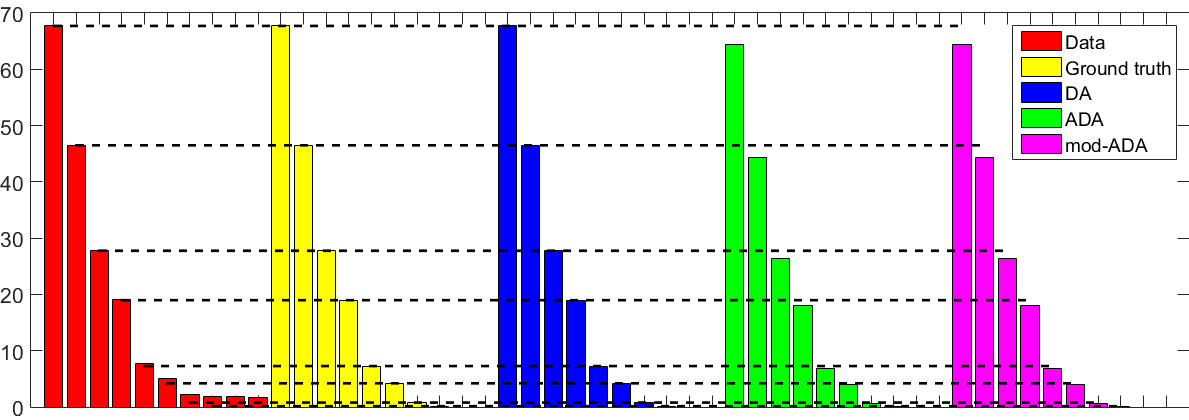}
\end{center}
\caption{The 10 leading singular values (averaged over 100 trials) for the noisy data, ground truth and the three tested methods for $\alpha=0.1$. }
\label{fig:singvals01}
\begin{center}
\includegraphics[width=90mm]{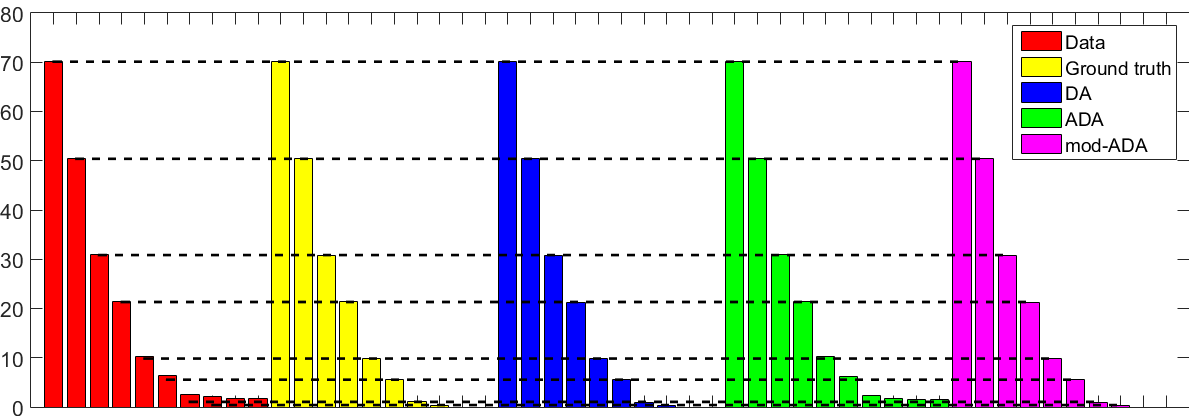}
\end{center}
\caption{The 10 leading singular values (averaged over 100 trials) for the noisy data, ground truth and the three tested methods for $\alpha=0.001$.}
\label{fig:singvals0001}
\end{figure}

In Figures~\ref{fig:singvals01} and \ref{fig:singvals0001} we show the 10 leading singular values of the solutions given by the tree methods (averaged over all trials).
For comparison we have also plotted the singular values of the measurement matrix and the ground truth (which is of rank $8$).
While convergence of ADA is relatively fast for $\alpha=0.1$ the added regularization term $\frac{\alpha}{2}\|X\|_2^2$ will penalize the larges singular values.
For ADA and mod-ADA this results in the weak shrinking bias visible in Figure~\ref{fig:singvals01}. For $\alpha=0.001$, see Figure~\ref{fig:singvals0001},
this bias is negligible but the slow convergence of ADA hinders the suppression of the small singular values.
Hence for ADA there is a trade-off between accuracy and speed of convergence.
The same tendency can be observed in Table~\ref{tab:gtdist} where we show the normalized distance to the ground truth, that is $\frac{\|H-H_{gt}\|_2^2}{\|H_{gt}\|_2^2}$, if $H_{gt}$ is the ground truth, averaged over all trials.
\begin{table}[htb]
\begin{center}
\begin{tabular}{c||c|c|c}
$\alpha$ & DA & ADA & mod-ADA\\
\hline \hline
0.1 & 0.0053 & 0.0479 & 0.0479 \\
0.001 & 0.0050  &  0.0753  &  0.0049
\end{tabular}
\end{center}
\caption{Average normalized distance to ground truth.}
\label{tab:gtdist}
\end{table}

According to the theory of Section \ref{subsec1} it may be necessary to select a subsequence to get convergence of DA. In Figure~\ref{fig:special_case_convergence} we highlight a single problem instance (extracted from the experiment above) where the variables did not seem to converge for DA. To the left we plot the primal and dual objective values and to the right we plot the distance between the primal solution $X^n$ and the ground truth during 300 iterations. The dual variable $\Lambda^n$ exhibits a similar behavior as $X^n$, note however that the dual objective values seem to converge nicely.
For comparison we also plot mod-ADA which does not exhibit the same behavior.

The exact reason for this behavior is unclear but it seems to happen for difficult problem instances, when the size of the 8th (true) singular value is at same level as of the noise.
Note that when using \eqref{eq:cossum} this may happen if for example $a$ is close to zero.
However the effects are not visible in Figure~\ref{fig:alphaconvergence} due to averaging.
\begin{figure}[htb]
\begin{center}
\includegraphics[width=60mm]{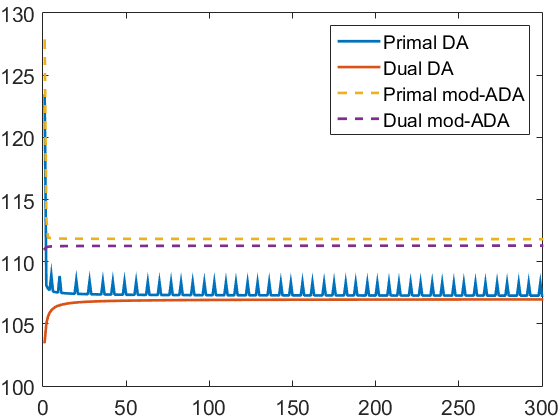}
\includegraphics[width=60mm]{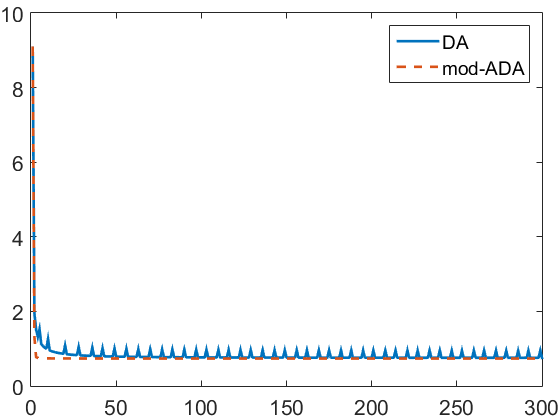}
\end{center}
\caption{Convergence of DA and mod-ADA for a difficult case. Left - Primal and dual values for 300 iterations. Right - Distance from current primal estimate to ground truth in each iteration.}
\label{fig:special_case_convergence}
\end{figure}

\subsection{Performance in signal frequency estimation}

In this section we compare our formulation to state-of-the-art frequency estimation methods.
\emph{Kroeneker's theorem} is the basis for many of these methods such as for instance ESPRIT \cite{esprit} and MUSIC \cite{MUSIC}.

We are interested in approximating a signal $f$ with a linear combination of $P$ exponentials.
We use the formulation
\begin{equation}\label{freqest_frob}
\argmin_{A \mbox{~is Hankel}} \sigma_0^2 \rank A + \| A - H(f)\|_2^2,
\end{equation}
where $\sigma_0$ is a parameter that penalizes the number of exponential functions used in the approximation. Note that due to the Hankel structure, the Frobenius norm formulation above  is equivalent (except some degenerate cases) to the weighted least squares objective
\begin{align*}
\min_a \quad & P \sigma_0^2 + \sum_{j=-N}^{N} w(j) \left| a(j)- f(j) \right|^2,\\
&\mbox{where } a(j) = \sum_{p=1}^{P} c_p e^{\zeta_p j}, \quad c_p,\zeta_p\in \mathbb{C},
\end{align*}
where $w$ is the triangle weight
$$
w(j) = N+1-|j|.
$$
This weight is undesirable, assuming that all signal samples are subjected to independent Gaussian noise, however removing it by modifying the Forbenius norm term makes primal updates much more difficult. Note that in contrast ESPRIT uses the unweighted least squares formulation, but without optimality guarantees.

\begin{table}[h]
\centering
\texttt{%
\caption{Frequencies and coefficients}
\label{tab:4exp}
\begin{tabular}{ | l | l | }
\hline
$\zeta_p$ & $c_p$ \\
\hline
5924.0i & +1.00000+i0.00000\\
804.24i        & +0.62348+i0.78183 \\
695.88i           & -0.22252+i0.97493 \\
7937.6i            &-0.90097+i0.43388\\
\hline
\end{tabular}
}
\end{table}
We conduct experiments on a function of the form \eqref{Kronecker} with frequencies $\zeta_k$ and coefficients $c_k$ given in Table \ref{tab:4exp}.
The function is sampled at 257 points, and white noise is added to achieve different signal-to-noise rations (SNR).
In the simulations below we tested SNR levels between 0 dBW and 25 dBW in steps of 2.5 dBW. For each SNR level 10000 simulations where computed using the dual ascent method and the ESPRIT method.
For the ESPRIT approach, exponentials are estimated using the $P=4$  largest singular vectors, and a least squares fit is used to determine the coefficients $c_p$ in  \eqref{Kronecker}.
A Hankel matrix is then generated from this vector and the approximation error is computed in Frobenius norm. For the dual ascent method, the penalty level is chosen as $\sigma_0 = (\sigma_4(F)+\sigma_5(F))/2$, where $F$ is the Hankel matrix from the noisy signal. The stepsize parameters $\alpha_k$ are chosen so that they decay as $k^{-1/2}$, and the algorithm is stopped when $\|A^{k}-\mathcal{P}_\mathcal{H}(A^k)\| < 10^{-6}$.

In the top left panel of Figure \ref{barplot}, the difference between the Frobenius errors obtained by the ESPRIT method and the dual ascent method are shown in bar plots. The differences between the errors are scaled with the noise level in order to illustrate them in the same plot, i.e., the bar plots illustrates
$$
\left( \| A_{\mathrm{ESPRIT}} - H(f)\| -\| A_{\mathrm{dual~ascent}} - H(f)\| \right) 10^{\mathrm{SNR}/20}.
$$
From Figure \ref{barplot} there are very few events on the lowest bar group, i.e., when the errors between the two methods are of equal size. In fact, the error obtained by the ESPRIT is higher than the one obtained by the dual ascent in every single one of the 110000 simulations. Note though that this result holds with respect to error in the Frobenius norm and not in the regular $\ell^2$-norm for the vectors. The lower left barplot of Figure \ref{barplot} shows the difference in $\ell^2$ norm on the vectors that generate the Hankel matrices, i.e.,
$$
\left( \| a_{\mathrm{ESPRIT}} - f\|_{\ell^2} -\| a_{\mathrm{dual~ascent}} - f\|_{\ell^2} \right) 10^{\mathrm{SNR}/20},
$$
where $a$ and $f$ are the vectors that generate the Hankel matrices $A$ and $F$, respectively. Here, we can see that the ESPRIT method typically gives a better approximation than the dual ascent.
\begin{figure*}[h!]
\centering
\includegraphics[width=0.49\linewidth, trim=0.0cm 0.0cm 0.0cm 0.0cm]{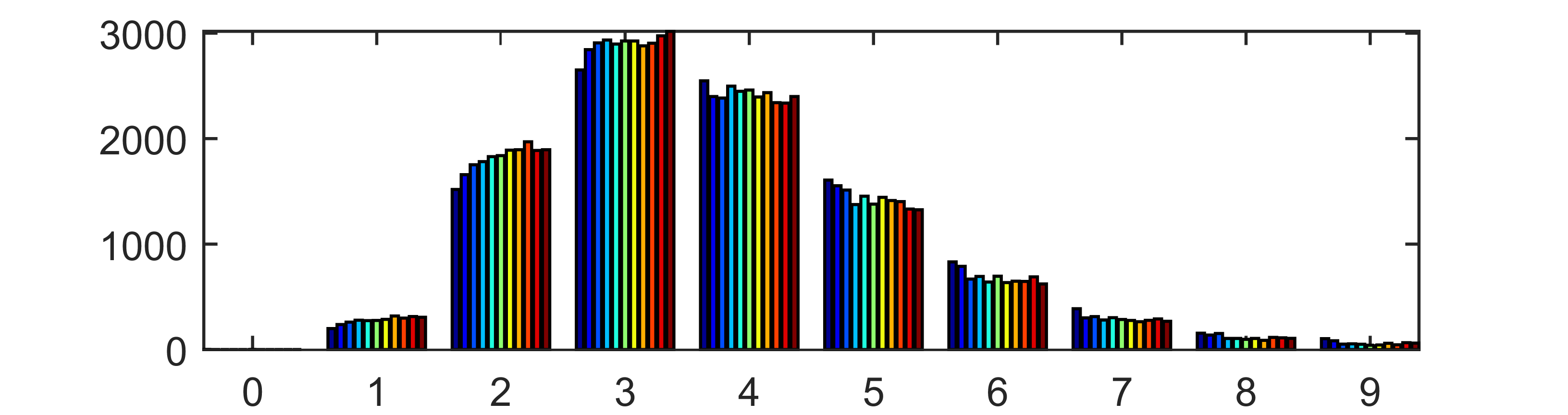}
\includegraphics[width=0.49\linewidth, trim=0.0cm 0.0cm 0.0cm 0.0cm]{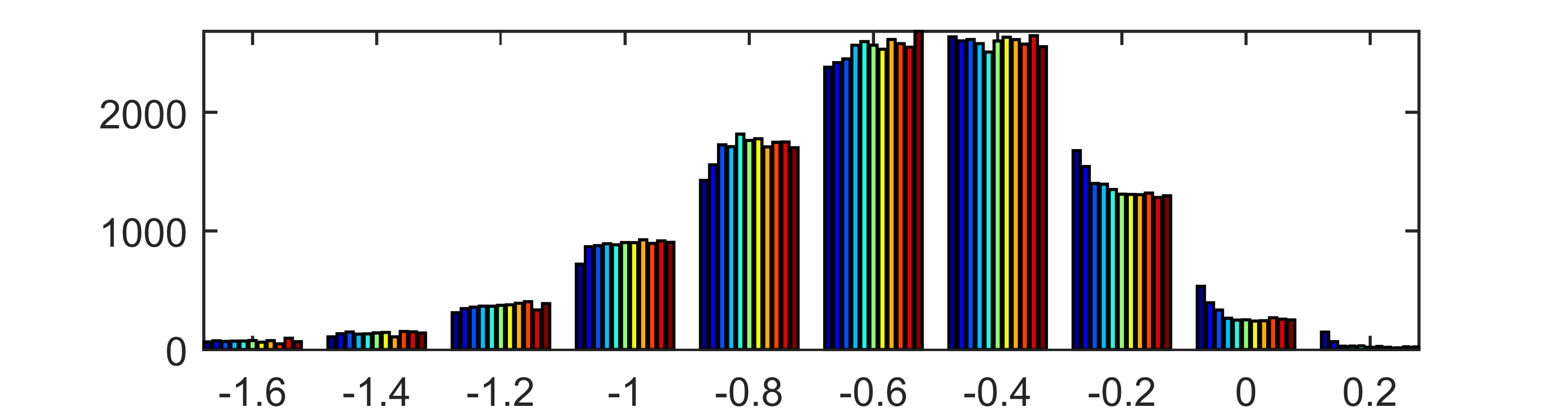}
\caption{\label{barplot} Barplots over the difference between the errors obtained by ESPRIT and dual ascent for simulations using 11 different SNR levels indicated by different colors. The left panel shows the difference for the Frobenius norm errors, while the right panel shows the difference for the $\ell^2$ errors. The errors have been normalized by the noise level. }
\end{figure*}



\section{Acknowledgment}
This research is partially supported by the Swedish Research Council, grants no. 2011-5589, 2012-4213 and 2015-03780; and the Crafoord Foundation.
\bibliographystyle{plain}
\bibliography{referenser,newlib}
\end{document}